\documentclass[12pt,reqno]{amsart}

\usepackage[T1]{fontenc}
\usepackage[utf8]{inputenc}
\usepackage[english]{babel}

\usepackage{amsfonts}
\usepackage{amsmath}
\usepackage{amssymb}
\usepackage{amsthm}
\usepackage{mathrsfs}
\usepackage{stmaryrd}
\usepackage{color}
\usepackage{url}
\usepackage{graphicx}
\usepackage[pdftex,colorlinks,citecolor=blue]{hyperref}
\usepackage{caption}
\usepackage{enumerate}
\usepackage{epstopdf}
\usepackage{hyphenat}
\usepackage{float}
\usepackage{indentfirst}
\usepackage[export]{adjustbox}
\usepackage{tikz}
\usepackage[all]{xy}
\usetikzlibrary{matrix,arrows,decorations.pathmorphing}
\usepackage{booktabs}
\usepackage{array}
\usepackage{bm}
\usepackage{csquotes}
\usepackage{enumerate}
\usepackage{comment}
\usepackage{multicol}
\usepackage{tikz-cd}

\newcolumntype{P}[1]{>{\centering\arraybackslash}p{#1}}

\allowdisplaybreaks

\usepackage[sorting=nyt]{biblatex}
\newtheorem{theorem}{Theorem}[section]

\newtheorem{lemma}[theorem]{Lemma}
\newtheorem{proposition}[theorem]{Proposition}
\newtheorem{corollary}[theorem]{Corollary}
\newtheorem{theoremx}{Theorem}

\theoremstyle{definition}

\newtheorem{assumption}[theorem]{Assumption}

\newtheorem{definition}[theorem]{Definition}

\theoremstyle{remark}
\newtheorem{remark}[theorem]{Remark}
\newtheorem{example}[theorem]{Example}

\DeclareFontFamily{U}{wncy}{}
\DeclareFontShape{U}{wncy}{m}{n}{<->wncyr10}{}
\DeclareSymbolFont{mcy}{U}{wncy}{m}{n}
\DeclareMathSymbol{\Sh}{\mathord}{mcy}{"58} 

\newcommand{\numberset}{\mathbb}

\newcommand{\C}{\numberset{C}}
\newcommand{\Q}{\numberset{Q}}
\newcommand{\Z}{\numberset{Z}}

\newcommand{\OF}{\mathcal{O}_F}

\newcommand{\V}{\boldsymbol{\mathrm{V}}}

\DeclareMathOperator{\Jac}{Jac}
\DeclareMathOperator{\Ta}{Ta}

\newcommand{\Taa}{\boldsymbol{\Ta}}
\newcommand{\T}{\boldsymbol{\mathrm{T}}}

\newcommand{\calR}{\mathcal{R}}
\newcommand{\A}{\boldsymbol{\mathrm{A}}}

\DeclareMathOperator{\sh}{Sh}
\DeclareMathOperator{\Hom}{Hom}

\DeclareMathOperator{\Char}{char}

\DeclareMathOperator{\ord}{ord}

\DeclareMathOperator{\Res}{res}
\DeclareMathOperator{\Gal}{Gal}

\DeclareMathOperator{\GL}{GL}
\DeclareMathOperator{\SL}{SL}

\DeclareMathOperator{\id}{id}

\DeclareMathOperator{\Tr}{Tr}

\DeclareMathOperator{\Aut}{Aut}
\DeclareMathOperator{\Frac}{Frac}

\DeclareMathOperator{\Cor}{cor}

\DeclareMathOperator{\res}{res}
\DeclareMathOperator{\loc}{loc}

\DeclareMathOperator{\Fr}{Fr}

\DeclareMathOperator{\hight}{\mathrm{ht}}

\DeclareMathOperator{\ks}{KS}
\newcommand{\KS}{\boldsymbol{\ks}}

\newcommand{\ac}{\mathrm{ac}}
\newcommand{\fs}{\mathrm{fs}}
\newcommand{\tr}{\mathrm{tr}}
\newcommand{\Gr}{\mathrm{Gr}}
\newcommand{\f}{\mathrm{f}}
\newcommand{\s}{\mathrm{s}}
\newcommand{\cyc}{\mathrm{cyc}}
\newcommand{\ta}{\mathrm{tame}}
\newcommand{\wi}{\mathrm{wild}}
\newcommand{\ur}{\mathrm{ur}}
\newcommand{\cont}{\mathrm{cont}}
\newcommand{\BK}{\mathrm{BK}}
\newcommand{\di}{\mathrm{div}}
\newcommand{\tors}{\mathrm{tors}}
\newcommand{\length}{\mathrm{length}}

\newcommand{\ga}{\alpha}
\newcommand{\gb}{\beta}

\newcommand{\e}{\varepsilon}

\newcommand{\gk}{\kappa}
\newcommand{\gl}{\lambda}

\newcommand{\gs}{\sigma}
\newcommand{\gt}{\tau}

\newcommand{\gL}{\Lambda}

\newcommand{\F}{\numberset{F}}
\newcommand{\Zp}{\numberset{Z}_p}

\newcommand{\m}{\mathfrak{m}}
\newcommand{\p}{\mathfrak{p}}

\newcommand{\h}{\mathfrak{h}}
\newcommand{\sfrak}{\mathfrak{s}}

\author{Francesco Zerman}
\address{UniDistance Suisse\\ Schinerstrasse 18, 3900, Brig, Switzerland}
\email{francesco.zerman@unidistance.ch}

\title{Quaternionic Kolyvagin systems and Iwasawa theory for Hida families}


\begin{document}

	\begin{abstract}
    	By applying a suitable Kolyvagin's descent, we perform the construction of a modified universal Kolyvagin system for the Galois representation attached to a Hida family of modular forms, starting from the big Heegner point Euler system of Longo--Vigni \cite{Longo-Vigni11:quaternion-algebras-Hida-families} built in towers of Shimura curves. We generalize the work of B{\"u}y{\"u}kboduk \cite{Buyukboduk:big-Heegner-point-kolyvagin-system} to a quaternionic setting, relaxing the classical \emph{Heegner hypothesis} on the tame conductor of the family. As a byproduct of this construction, we give a proof of one divisibility of the anticyclotomic Iwasawa main conjecture for Hida families.
	\end{abstract}

    \maketitle

\section{Introduction}

Let $N$ be a positive integer, $p\nmid 6N$ be a prime number and $\omega\colon (\Z/p\Z)^\times\to \boldsymbol{\mu}_{p-1}$ be the Teichm\"uller character. Let 
\begin{equation*}
	f(q)=\sum_{n=1}^\infty a_n(f)q^n\in S_k\big(\Gamma_0(Np),\omega^j\big)
\end{equation*}
be an ordinary $p$-stabilized newform, where $j\equiv k\bmod 2$. Assume that the residual representation attached to $f$ is absolutely irreducible and $p$-distinguished (see Assumption \ref{ass:residual-representation-irreducible}).

In this paper, we are interested in the study of some arithmetic invariants of the Galois representation attached to the $p$-adic \emph{Hida family} passing through $f$. Thanks to the work of Hida \cite{Hida86:Galois-representations, Hida86:Iwasawa-modules}, this is a free rank-$2$ module $\T$ over a $2$-dimensional complete local Noetherian ring $\calR$ of residue characteristic $p$, endowed with a linear action of the absolute Galois group $G_\Q$ of $\Q$. This \emph{big} representation is related to the Galois representations attached to the cusp forms that lie in the Hida family via \emph{arithmetic specializations}, as we briefly recall in Section \ref{sec:arithmetic-primes}.

Fix an imaginary quadratic field $K$ of discriminant prime to $Np$ and class number prime to $p$. We assume that $K$ satisfies the \emph{generalized Heegner hypothesis} with respect to $N$ (see Assumption \ref{ass:field-K}), i.e.~that $N=N^+ N^-$ where the primes dividing $N^+$ are split in $K$ and $N^-$ consists in the square-free product of an even number of primes that are inert in $K$. In this setting, Longo--Vigni \cite{Longo-Vigni11:quaternion-algebras-Hida-families} built a set of \emph{big Heegner classes}
\begin{equation*}
    \gk_c\in H^1(K[c],\T^\dag)
\end{equation*}
for every positive integer $c$ coprime with $N$, where $K[c]$ is the ring class field of $K$ of conductor $c$ and $\T^\dag$ is the self-dual twist of the representation $\T$ (see Definition \ref{def:self-dual-twist}). These classes satisfy some rigidity properties (see \cite[§8]{Longo-Vigni11:quaternion-algebras-Hida-families}) and can be easily modified (see Definition \ref{def:classes-zeta}) into an anticyclotomic Euler system (or, more precisely, into a \emph{$p$-complete} anticyclotomic Euler system in the sense of \cite{mastella-zerman2025:anticyclotomic}).

One of the goals of this paper is to show how one can use these classes to study some conjectures on the anticyclotomic Iwasawa theory for $\T^\dag$. Let $K_\infty/K$ be the anticyclotomic $\Z_p$-extension of $K$, set $\Gamma^\ac:=\Gal(K_\infty/K)$ and denote by $K_t/K$ the subextension of $K_\infty/K$ of degree $p^t$, for every $t\ge 0$. One of the main topics of Iwasawa theory is the study of the structure of \emph{Iwasawa Selmer groups}, that are submodules of $\varprojlim_t H^1(K_t,\T^\dag)$. By Shapiro's lemma (see Section \ref{sec:shapiro-lemma}), we shift the above problem to the study of the arithmetic of the universal anticyclotomic twist $\T^\ac:=\T^\dag\otimes_\calR\calR\llbracket\Gamma^\ac\rrbracket$ over $K$.

Under some technical assumptions on the ramification of $\T^\dag$ at the primes dividing $Np$ (see Assumptions \ref{ass:tamagawa-factors-at-N} and \ref{ass:H.stz}) together with a big image assumption (Assumption \ref{ass:big-image}), in Section \ref{sec:the-big-heegner-point-kolyvagin-system} we pursue a suitable Kolyvagin's descent to the big Heegner classes $\kappa_c$. This process, inspired by the seminal works of Kolyvagin \cite{Kolyvagin:euler-systems, Kolyvagin:finiteness-of-E-and-sha, Kolyvagin-Logachev:finiteness-of-Sha} and later refined in the anticyclotomic setting by Howard \cite{Howard04:heegner-point-kolyvagin-system} for elliptic curves, Nekov{\'a}{\v{r}} \cite{Nekovar92:chow-groups} and Longo--Vigni \cite{Longo-Vigni:kolyvagin-systems-Heegner-cycles} for modular forms and B{\"u}y{\"u}kboduk \cite{Buyukboduk:big-Heegner-point-kolyvagin-system} for non-quaternionic families of modular forms, gives rise to an explicit set of elements (see Definition \ref{def:derived-classes})
\begin{equation}\label{eq:introduction-classes}
    \boldsymbol{\gk}(n)_{t,m,s}\in H^1(K,\T_{t,m,s}),
\end{equation}
defined over some finite quotients $\T_{t,m,s}$ of $\T^\ac$ (defined in Section \ref{sec:quotients-galois-groups-and-primes}) and with $n$ varying in the set of square-free products of primes $\ell$ that are inert in $K$, coprime with $Np$ and such that the conjugacy class in $\Gal(K(\T_{t,m,s})/\Q)$ of the arithmetic Frobenius $\Fr_\ell$ at $\ell$ coincides with the conjugacy class of the complex conjugation.

The main result of this paper is to show that a slight perturbation of these classes forms a modified universal Kolyvagin system. This notion, studied in detail in \cite{mastella-zerman2025:anticyclotomic}, is a generalization of the classical notion of (universal) Kolyvagin system, first defined in \cite{Mazur-Rubin:Kolyvagin-systems} and later studied by many other authors, e.g.~\cite{Howard04:heegner-point-kolyvagin-system, Buyukboduk:deformations-of-kolyvagin-systems}. Let $\mathcal{F}_\Gr$ be the Greenberg Selmer structure on $\T^\ac$ over $K$ (see Definition \ref{dfn:strict-Greenberg-selmer-structure}), $\mathcal{P}'$ be a suitable set of primes that are inert in $K$ and denote by $\mathcal{N}'$ the set of square-free products of elements of $\mathcal{P}'$. For every couple $(n,\ell)$ with $n\ell\in\mathcal{N}'$ and $\ell$ prime, let $\chi_{n,\ell}$ be an automorphism of $\T^\ac$. In Section \ref{sec:selmer-structures-and-kolyvagin-systems} we recall the definition of the module of (modified) universal Kolyvagin systems for the quadruple $(\T^\ac,\mathcal{F}_\Gr,\mathcal{P}',\{\chi_{n,\ell}\})$, to be denoted by $\overline{\KS}(\T^\ac,\mathcal{F}_\Gr,\mathcal{P}',\{\chi_{n,\ell}\})$. The automorphisms $\{\chi_{n,\ell}\}$ are used to modify the classical notion of universal Kolyvagin system (see \cite[Definition 3.4]{Buyukboduk:deformations-of-kolyvagin-systems}). Then, the main result of this paper reads as follows.

\begin{theoremx}[Theorem \ref{thm:main-thm}]\label{thm:A}
    Under Assumptions \ref{ass:residual-representation-irreducible}, \ref{ass:field-K}, \ref{ass:tamagawa-factors-at-N}, \ref{ass:H.stz} and \ref{ass:big-image}, there is an infinite set of primes $\mathcal{P}'$, a set of automorphisms $\{\chi_{n,\ell}\}$ of $\T^{\ac}$ and a universal Kolyvagin system $\boldsymbol{\gk}^{\ac}\in\overline{\KS}(\T^{\ac},\mathcal{F}_{\Gr},\mathcal{P}',\{\chi_{n,\ell}\})$ such that
    \begin{equation*}
        \boldsymbol{\gk}^{\ac}(1)=\varprojlim_t\big(\Cor^{K[p^{t+1}]}_{K_t}U_p^{-t}\gk_{p^{t+1}}\big)\in H^1(K,\T^{\ac}),
    \end{equation*}
    where $U_p\in\calR^\times$ is the $p$-th Hecke operator.
\end{theoremx}

Sections \ref{sec:kolyvagin-primes-and-main-result}--\ref{sec:the-key-relation} are devoted to the proof of this theorem. There, we explicitly define the set $\mathcal{P}'$ of admissible primes (see Definition \ref{dfn:kolyvagin-primes}), the automorphisms $\chi_{n,\ell}$ and the modified universal Kolyvagin system $\boldsymbol{\gk}^{\ac}$ (see Section \ref{sec:the-key-relation}), built out of the classes \eqref{eq:introduction-classes}. Usually, it is not possible to massage a modified universal Kolyvagin system into a classical universal Kolyvagin system, and we don't have any evidence that this can be done in our setting. However, in Section \ref{sec:Anticyclotomic-Iwasawa-theory}, we show that there is no detectable difference between the two notions when it comes to applications.

Indeed, in the last section we explain how the existence of a non-trivial modified universal Kolyvagin system yields to the proof of one divisibility of the big Heegner point main conjecture for $\T^\ac$ (\cite[Conjecture 10.8]{Longo-Vigni11:quaternion-algebras-Hida-families}) if we assume that $\calR$ is regular (Assumption \ref{ass:R-regular}). Let $\calR^\ac:=\calR\llbracket\Gamma^\ac\rrbracket$ and $\A^\ac:=\T^\ac\otimes_{\calR^\ac}(\calR^\ac)^\vee$, where $(\bullet)^\vee:=\Hom_{\cont}(\bullet,\Q_p/\Z_p)$ is the Pontryagin dual. For $X$ equal to $\T^\ac$ or $\A^\ac$, denote by $H^1_{\mathcal{F}_\Gr}(K,X)$ the Greenberg Selmer module of $X$ over $K$. If $M$ is a finitely generated torsion $\calR^\ac$-module, its \emph{characteristic ideal}
\begin{equation*}
         \Char_{\calR^{\ac}}(M)=\prod_{\hight\p=1} \p^{\length_{\calR^\ac_\p}(M_\p)}\subseteq \calR^\ac
\end{equation*}
classifies $M$ up to pesudo-isomorphism, where the product runs over the set of high-one primes of $\calR^\ac$. Our second main result reads as follows.

\begin{theoremx}[Theorem \ref{thm:iwasawa-main-conjecture}]\label{thm:B}
    Assume that the system $\boldsymbol{\gk}^\ac\in \overline{\KS}(\T^\ac,\mathcal{F}_\Gr,\mathcal{P}',\{\chi_{n,\ell}\})$ of Theorem \ref{thm:A} is such that $\boldsymbol{\gk}^{\ac}(1)\ne 0$. Then  
    \begin{itemize}
        \item[(a)] $H^1_{\mathcal{F}_\Gr}(K,\T^\ac)$ is a torsion-free  $\calR^\ac$-module of rank 1.
        \item[(b)] $\Char_{\calR^\ac}(H^1_{\mathcal{F}_\Gr}\big(K,\A^\ac)^\vee_{\tors}\big)\supseteq \Char_{\calR^\ac}\big(H^1_{\mathcal{F}_\Gr}(K,\T^{\ac})/\calR^\ac\cdot \boldsymbol{\gk}^\ac(1)\big)^2$.
    \end{itemize}
\end{theoremx}

The proof of this theorem follows closely the arguments of Fouquet \cite[§5-6]{Fouquet13:Dihedral}. If we further assume that $p\nmid\varphi(N)$, the work of Cornut--Vatsal \cite{cornut-vatsal2007:nontriviality} (see also \cite[Corollary 3.1.2]{Howard07:variation-of-Heegner-points-in-Hida-families}) implies that $\boldsymbol{\gk}^\ac(1)\ne 0$, therefore the statement of Theorem \ref{thm:B} simplifies. 

\subsection{Acknowledgments}
I would like to thank Stefano Vigni, who gave me this project for my Ph.D. thesis and supported me in every moment. A special thanks goes also to Matteo Longo, for his fast and valuable answers, and to Luca Mastella, for interesting discussions on these topics. The author was supported by the ERC Consolidator Grant \emph{ShimBSD: Shimura varieties and the BSD conjecture}.

\section{Big Galois representations and big Heegner classes}

In this section we study the big Galois representation attached to a Hida family of cusp forms and we review the definition of the big Heegner classes of \cite{Longo-Vigni11:quaternion-algebras-Hida-families}, which live in the cohomology of the self-dual twist of this representation.

\subsection{Setting} Let $N^-$ and $N^+$ be coprime positive integers, with $N^-$ consisting of the square-free product of an even number of primes. Let $N=N^+N^-$ and fix a prime $p\nmid 6N$. Fix once and for all embeddings of algebraic closures $\bar{\Q}\hookrightarrow \bar{\Q}_p$ and $\bar{\Q}\hookrightarrow\C$. Call $\boldsymbol{\mu}_{p-1}$ the group of $(p-1)$-st roots of unity. Define $\omega\colon (\Z/p\Z)^\times\to \boldsymbol{\mu}_{p-1}$ to be the Teichm\"uller character, where we see $\boldsymbol{\mu}_{p-1}$ both inside $\Z_p^\times$ and $\bar{\Q}$. Let
	\begin{equation*}
	f(q)=\sum_{n=1}^\infty a_n(f)q^n\in S_k\big(\Gamma_0(Np),\omega^j\big)
	\end{equation*}
	be a normalized cusp form of weight $k\ge 2$ and character $\omega^j$, which is an eigenform for the Hecke operators $T_\ell$ for $\ell\nmid Np$ and $U_\ell$ for $\ell\mid Np$. Here $j$ is an integer that necessarily satisfies $j\equiv k\pmod 2$. Let $F$ be a finite extension of $\Q_p$ which contains all Fourier coefficients of $f$ and let $\OF$ be its valuation ring. We assume further that $f$ is an ordinary $p$-stabilized newform, in the sense that $a_p(f)\in\OF^\times$ and that the conductor of $f$ is divisible by $N$ (see \cite[Section 2.4]{Vigni22:shafarevich-tate-groups}). 

Let $T_f$ be a $G_\Q$-stable $\OF$-lattice inside the Deligne representation $V_f$ attached to $f$ (see \cite{Deligne:formes-modulaires}), which induces a representation
\begin{equation*}
    \rho_{f}\colon G_\Q\longrightarrow \GL(T_{f})\cong\GL_2(\OF)\subseteq\GL_2(F)
\end{equation*}
of the absolute Galois group of $\Q$ with the property that the arithmetic Frobenius $\Fr_\ell$ at any prime $\ell\nmid Np$ has characteristic polynomial
 \begin{equation*}
     X^2-a_\ell(f) X+\omega^j(\ell)\ell^{k-1}.
 \end{equation*}
If $\pi_F$ is a uniformizer for $\OF$, the \emph{residual $G_\Q$-representation of $\rho_f$} is the composition
 \begin{equation*}
            \bar{\rho}_f: G_\Q\longrightarrow \GL(T_f)\longrightarrow \GL(T_f/\pi_F T_f).
\end{equation*}
The representation $\bar{\rho}_f$ is said to be \emph{$p$-distinguished} if its restriction to a decomposition group $\mathfrak{D}_p$ at $p$ can be put in the shape $\bar{\rho}_f|_{\mathfrak{D}_p}=\left(\begin{smallmatrix}
        \e_1 &*\\ 0&\e_2
    \end{smallmatrix}\right)$ for characters $\e_1\ne \e_2$ (see \cite[Definition 4]{Ghate05:ordinary-forms}). Notice that, since $f$ is $p$-ordinary, it is a result due to Mazur and Wiles that $\bar{\rho}_f|_{\mathfrak{D}_p}$ is always upper triangular (see \cite[§1]{Ghate05:ordinary-forms}).

 \begin{assumption}\label{ass:residual-representation-irreducible}
     Assume that the residual $G_\Q$-representation $\bar{\rho}_{f}$ is irreducible and $p$-distinguished. 
 \end{assumption}

 Under this assumption, all $G_\Q$-stable $\OF$-lattices in $V_f$ are homothetic, therefore the representation $T_f$ is uniquely determined up to isomorphism. Since $p\ne 2$, the representation $\bar{\rho}_{f}$ is irreducible if and only if it is absolutely irreducible (see \cite[Remark 2.5]{Vigni22:shafarevich-tate-groups}).

 \begin{assumption}\label{ass:field-K}
 Fix an imaginary quadratic field $K$ different from $\Q(\sqrt{-1})$ and $\Q(\sqrt{-3})$ with the following properties:
		\begin{itemize}
			\item[(a)] The class number of $K$ is prime to $p$.
			\item[(b)] The discriminant of $K$ is prime to $Np$.
			\item[(c)] All primes dividing $N^+$ (respectively, $N^-$) are split (respectively, inert) in $K$.
		\end{itemize}
\end{assumption}

\begin{remark}
    Points (a) and (b) are assumed in order to simplify many technical arguments. For instance, point (a) implies that the anticyclotomic $p$-extension of $K$ is totally ramified at the primes of $K$ above $p$ (see Section \ref{sec:anticyclotomic-twists}) and point (b) is used in the proof of Lemma \ref{lem:no-invariants}. Point (c) is sometimes called the \emph{generalized Heegner hypothesis} for $K$ with respect to $N=N^+N^-$, which is needed to build the big Heegner points of \cite{Longo-Vigni11:quaternion-algebras-Hida-families} (see Section \ref{sec:big-heegner-classes}). We believe that many arguments that appear in this paper can be adapted to the case when the class number of $K$ is not coprime with $p$ or when $p$ is ramified in $K$.
\end{remark}

Assumptions \ref{ass:residual-representation-irreducible} and \ref{ass:field-K} will be in force until the end of the paper.

\subsection{Hida families} For the convenience of the reader, we recall the construction of the branch of the Hida family where $f$ belongs. For simplicity, we mainly follow \cite[§2.1]{Howard07:variation-of-Heegner-points-in-Hida-families}, \cite[§5]{Longo-Vigni11:quaternion-algebras-Hida-families} and \cite[§1.4]{Nekovar-Plater:parity-selmer-groups}. The reader may also refer to \cite{Hida86:Galois-representations, Hida86:Iwasawa-modules} and \cite[§12.7]{Nekovar06:selmer-complexes} for more details. To start with, let's abbreviate 
\begin{equation*}
    \Gamma:=1+p\Z_p\quad\text{and}\quad \gL_F:=\OF\llbracket \Gamma\rrbracket\subseteq\OF\llbracket \Z_p^\times\rrbracket
\end{equation*}
and fix a profinite generator $\gamma$ of $\Gamma$. For every $r\ge 2$ and $m\ge 1$, let $\mathfrak{h}_{r,m}$ be the $\OF$-algebra generated by all Hecke operators $T_\ell$ for $\ell\nmid Np$, together with the operators $U_\ell$ for $\ell\mid Np$ and the diamond operators $\langle d\rangle$ for $d\in (\Z/p^m\Z)^\times$, acting on the space of cusp forms $S_r(\Gamma_0(N)\cap\Gamma_1(p^m),\bar{\Q}_p)$. We make $\mathfrak{h}_{r,m}$ into an $\OF\llbracket \Z_p^{\times}\rrbracket$-algebra via the map
\begin{equation*}
    [z]\longmapsto z^{r-2}\langle z\rangle
\end{equation*}
for every $z\in\Z_p^\times$, where we indicate with square brackets the group elements in $\OF\llbracket \Z_p^{\times}\rrbracket$. Hida's ordinary projector $e^{\ord}:=\lim_s U_p^{s!}$ defines an idempotent in each $\mathfrak{h}_{r,m}$ and these are compatible with respect to the natural surjections $\h_{r,m+1}\twoheadrightarrow\h_{r,m}$. If we define $\h_{r,m}^{\ord}:=e^{\ord}\h_{r,m}$, the algebra
\begin{equation*}
    \h_r^{\ord}:=\varprojlim_m \h_{r,m}^{\ord}
\end{equation*}
is finite over $\gL_F$ and the image of $U_p$ is invertible in $\h_r^{\ord}$.

 \begin{definition}
        For every $r\ge 2$ and $m\ge 1$, define 
        \begin{itemize}
            \item $\omega_{r,m}=[\gamma]^{p^{m-1}}-\gamma^{(r-2)p^{m-1}}\in \gL_F$;
            \item $P_{r,\e}=[\gamma]-\e(\gamma)\gamma^{r-2}\in\mathcal{O}_F[\e]\llbracket \Gamma\rrbracket$ for every homomorphism $\e\colon \Gamma/\Gamma^{p^{m-1}}\to\bar{\Q}_p^\times$.
        \end{itemize}
    \end{definition}
    An elementary computation shows that $\omega_{r,m}=\prod_{\e}P_{r,\e}$, where the product is taken over all homomorphisms $\e: \Gamma/\Gamma^{p^{m-1}}\to\bar{\Q}_p^\times$. We now recall the following structure theorem about big ordinary Hecke algebras.

    \begin{theorem}[Hida]\label{thm:first-Hida-structure-theorem}
    Let $r,r'\ge 2$.
    \begin{itemize}
        \item[(a)] There is a canonical isomorphism of $\gL_F$-algebras $\mathfrak{h}_r^{\ord}\cong \mathfrak{h}_{r'}^{\ord}$ that sends Hecke operators $T_\ell$, $U_\ell$ and $d^{r-2}\langle d\rangle$ to $T_\ell$, $U_\ell$ and $d^{r'-2}\langle d\rangle$, respectively. 
        \item[(b)] The canonical map $\mathfrak{h}_2^{\ord}\overset{\cong}{\to}\mathfrak{h}_{r}^{\ord}\to \mathfrak{h}_{r,m}^{\ord}$ induces an isomorphism
    \begin{equation*}
        \mathfrak{h}_2^{\ord}/\omega_{r,m}\mathfrak{h}_2^{\ord}\overset{\cong}{\longrightarrow} \mathfrak{h}_{r,m}^{\ord}
    \end{equation*}
    for every $m\ge 1$.
    \end{itemize}
    \end{theorem}
    \begin{proof}
        See \cite[Proposition 1.4.3]{Nekovar-Plater:parity-selmer-groups}.
    \end{proof}
    From now on, we will then work with the algebra $\h^{\ord}:=\h_2^{\ord}$. By duality, the fixed cuspform $f$ determines an $\OF$-algebra map
    \begin{equation}\label{eq:theta-f}
        \theta_f\colon \h^{\ord}\twoheadrightarrow \h_{k,1}^{\ord}\to\OF
    \end{equation}
    characterized by $\theta_f(T_\ell)=a_\ell(f)$ for $\ell\nmid Np$,  $\theta_f(U_\ell)=a_\ell(f)$ for $\ell\mid Np$ and $\theta_f([\gamma])=\gamma^{k-2}$. By \cite[Corollary 7.6]{Eisenbud:commutative-algebra}, the $\gL_F$-algebra $\h^{\ord}$ decomposes as the direct sum of its completions at maximal ideals, and we let $\h_{\m_f}^{\ord}$ be the unique local summand through which $\theta_f$ factors. It is also known (see \cite[§12.7.5]{Nekovar06:selmer-complexes} or \cite[Corollary 1.4]{Hida86:Galois-representations}) that the localization of $\h_{\m_f}^{\ord}$ at $\ker(\theta_f)$ is a discrete valuation ring, unramified over the localization of $\gL_F$ at $\gL_F\cap \ker(\theta_f)$. Therefore, there is a unique minimal prime $\mathfrak{a}_f$ contained in $\ker(\theta_f)$. We set $R_f=\h_{\m_f}^{\ord}/\mathfrak{a}_f$.

    \begin{definition}
        The integral closure $\mathcal{R}$ of $R_f$ is called the \emph{branch} of the Hida family where $f$ lives. We denote by $\m_{\mathcal{R}}$ the maximal ideal of $\mathcal{R}$.
    \end{definition}

    \begin{remark}
        This notation is not completely settled, as some authors, such as Howard \cite{Howard07:variation-of-Heegner-points-in-Hida-families}, use the name ``branch of the Hida family where $f$ lives'' for the ring $R_f$. Our choice of terminology, instead, follows \cite{Longo-Vigni11:quaternion-algebras-Hida-families}.
    \end{remark}

    \begin{lemma}
   \label{lem:algebraic-properties-of-R}
       The ring $\mathcal{R}$ is a complete local Noetherian domain which is finitely generated over $\gL_F$. Moreover, it is a Cohen-Macaulay ring, free over $\gL_F$.
   \end{lemma}
   \begin{proof}
       For the first part, see \cite[Proposition 5.2]{Longo-Vigni11:quaternion-algebras-Hida-families}. Since $\mathcal{R}$ has Krull dimension $2$ and is integrally closed, Serre's criterion for normality implies that it is a Cohen-Macaulay ring. By miracle flatness, we conclude that $\mathcal{R}$ is flat (hence free) over $\gL_F$.
   \end{proof}

\subsection{Arithmetic primes}\label{sec:arithmetic-primes}
        Let $A$ be a finite commutative $\gL_F$-algebra. An $\OF$-algebra map $A\to\bar{\Q}_p$ is called \emph{arithmetic} if the composition
        \begin{equation*}
            \begin{tikzcd}
	\Gamma & {A^\times} & {\bar{\Q}_p^\times}
	\arrow["{\gamma\mapsto [\gamma]}", from=1-1, to=1-2]
	\arrow[from=1-2, to=1-3]
        \end{tikzcd}
        \end{equation*}
        has the form $\gamma\mapsto \psi(\gamma)\gamma^{r-2}$ for some $r\ge 2$ and some finite order character $\psi$ of $\Gamma$. The kernel of an arithmetic map is called an \emph{arithmetic prime} of $A$.

    As noted in \cite[§12.7.2 and §12.7.4]{Nekovar06:selmer-complexes}, the arithmetic primes of $A$ are exactly all primes lying above $P_{r,\e}\gL_{F'}\cap\gL_F$ for some $r\ge 2$, $F'$ finite extension of $F$ and $\e:\Gamma\to \mathcal{O}_{F'}$ finite order character. 

    Given an arithmetic prime $\p$, the \emph{residue field} $F_\p:=A_\p/\p A_\p=\Frac(A/\p)$ is a finite extension of $F$. The composition $\Gamma\to A^\times\to F_\p^\times$ has the form $\gamma\mapsto \psi_\p(\gamma)\gamma^{r_\p-2}$ for a finite order character $\psi_\p\colon\Gamma\to F_\p^\times$, called the \emph{wild character of $\p$}, and an integer $r_\p\ge 2$, called the \emph{weight of $\p$}.

    \begin{example}
        The map $\theta_f\colon \h_{\m_f}^{\ord}\to\OF$ induced by \eqref{eq:theta-f} is an arithmetic map with trivial wild character and weight $k$.
    \end{example}

    \begin{theorem}[Hida]\label{thm:main-thm-Hida-families}
       Let $\p$ be an arithmetic prime of $\mathcal{R}$ of weight $r_\p$ and wild character $\psi_\p$. Set $m_\p$ to be the maximum between $1$ and the $p$-adic order of the conductor of $\psi_\p$. Then the composition
        \begin{equation*}
            \h^{\ord}\longrightarrow\h_{\m_f}^{\ord}\longrightarrow\mathcal{R}\longrightarrow F_\p
        \end{equation*}
        factors through $\h_{r_\p,m_\p}^{\ord}$ and determines by duality an ordinary $p$-stabilized newform
        \begin{equation*}
            f_\p\in S_{r_\p}(\Gamma_0(Np^{m_\p}),\psi_\p\omega^{k+j-r_\p},F_\p).
        \end{equation*}
    \end{theorem}
    \begin{proof}
        See \cite[p.300]{Longo-Vigni11:quaternion-algebras-Hida-families}. See also \cite[p.97]{Howard07:variation-of-Heegner-points-in-Hida-families} and \cite[§12.7.4 and §12.7.5]{Nekovar06:selmer-complexes}.
    \end{proof}

\subsection{Big Galois representations}
   
 For every integer $m\ge 0$, denote by $X_m$ the compactified modular curve of level structure $\Gamma_0(N)\cap\Gamma_1(p^m)$ viewed as a scheme over $\Q$, by $\Jac(X_m)$ its Jacobian variety and by $\Ta_p(\Jac(X_m))$ the $p$-adic Tate module of the Jacobian. All Hecke and diamond operators act on $\Jac(X_m)$ and on $\Ta_p(\Jac(X_m))$ via the Albanese action. There is also a natural action of $G_\Q$ on them, coming from the fact that the curve $X_m$ is defined over $\Q$.
    
    As in \cite[§2.1]{Howard07:variation-of-Heegner-points-in-Hida-families} and \cite[§5.5]{Longo-Vigni11:quaternion-algebras-Hida-families}, for every integer $m\ge 1$ we define the $\h^{\ord}$-modules
    \begin{equation*}
        \begin{split}
            \Taa^{\ord}_{\m_f}&=\bigg(\varprojlim_m e^{\ord}\big(\Ta_p(\Jac(X_m))\otimes_{\Z_p}\OF \big)\bigg)\otimes_{\h^{\ord}}\h_{\m_f}^{\ord}\quad\text{and}\quad \T:=\Taa^{\ord}_{\m_f}\otimes_{\h_{\m_f}^{\ord}}\mathcal{R},
        \end{split}
    \end{equation*}
    where the inverse limit is taken with respect to the maps induced by the degeneracy maps $X_{m+1}\to X_m$. These two modules are endowed with a natural $\h^{\ord}$-linear action of the Galois group $G_\Q$. We will work with a twist of the representation $\T$, which we now recall from \cite[p. 96]{Howard07:variation-of-Heegner-points-in-Hida-families} and \cite[§5.4]{Longo-Vigni11:quaternion-algebras-Hida-families}.

    \begin{definition}\label{def:self-dual-twist}
        Factor the $p$-adic cyclotomic character $\e_{\cyc}\colon G_\Q\to \Z_p^\times$ as the product $\e_{\cyc}=\e_{\ta}\cdot \e_{\wi}$ with $\e_{\ta}\colon G_\Q\to \boldsymbol{\mu}_{p-1}$ and $\e_{\wi}\colon G_\Q\to \Gamma$. Define the \emph{critical character} $\Theta\colon G_\Q\to \gL_F^\times$ by
        \begin{equation*}
            \Theta:=\e_{\ta}^{\frac{k+j}{2}-1} \left[\e_{\wi}^{1/2}\right],
        \end{equation*}
        where $\e_{\wi}^{1/2}$ is the unique square root of $\e_{\wi}$ taking values in $\Gamma$.
    \end{definition}
    As shown in \cite[(2)]{Howard07:variation-of-Heegner-points-in-Hida-families}, the equality $\Theta^2=[\e_{\cyc}]$ holds in $\mathcal{R}$. This implies that, if $\ell$ is a prime different from $p$, the relation $\Theta^2(\Fr_\ell)=[\ell]$ holds in $\mathcal{R}$, where $\Fr_\ell$ is any arithmetic Frobenius at $\ell$.

    \begin{definition}
         Let $\mathcal{R}^\dag$ denote $\mathcal{R}$ as a module over itself with $G_\Q$ acting through the character $\Theta^{-1}$. The \emph{critical twist} of $\T$ is the $G_\Q$-module
        \begin{equation*}
            \T^{\dag}:=\T\otimes_{\mathcal{R}}\mathcal{R}^{\dag}.
        \end{equation*}
    \end{definition}

    \begin{proposition}\label{prop:characteristic-polynomial-Frobenius}
		The $\mathcal{R}$-module $\T^{\dag}$ is free of rank two. As a $G_\Q$-representation, $\T^{\dag}$ is unramified outside $Np$. The arithmetic Frobenius $\Fr_\ell$ of a prime $\ell\nmid Np$ acts on $\T^{\dag}$ with characteristic polynomial
		\begin{equation*}
		X^2-\Theta^{-1}(\Fr_\ell)T_\ell X+\ell.
		\end{equation*}
	\end{proposition}
	\begin{proof}
        The first claim follows immediately from \cite[Proposition 2.1.2]{Howard07:variation-of-Heegner-points-in-Hida-families}. The second follows from \cite[Proposition 2.1.2]{Howard07:variation-of-Heegner-points-in-Hida-families} and the fact that the character $\Theta$ is unramified outside $p$.
		
		In order to prove the third claim, we start from the fact that the action of $\Fr_\ell$ on $\Taa^{\ord}_{\m_f}$ has characteristic polynomial $X^2-T_\ell X+[\ell]\ell$, as explained in \cite[Proposition 2.1.2]{Howard07:variation-of-Heegner-points-in-Hida-families}. When we move to $\T^\dag$, the original action of $G_\Q$ becomes twisted by $\Theta^{-1}$. Elementary computations show that the action of $\Fr_\ell$ has characteristic polynomial
		\begin{equation*}
		X^2-\Theta^{-1}(\Fr_\ell)T_\ell X+\Theta^{-2}(\Fr_\ell)[\ell]\ell
		\end{equation*}
	on $\T^\dag$. Since $\Theta^2(\Fr_\ell)=[\ell]$, we obtain the claimed polynomial.
	\end{proof}

        \begin{lemma}\label{lem:twist-Hida-representation-irreducible}
            The residual $G_\Q$-representation $\bar{\T}^\dag:=\T^\dag/\m_{\mathcal{R}}\T^\dag$ is absolutely irreducible.
        \end{lemma}
        \begin{proof}
            By \cite[Proposition 5.4]{Longo-Vigni11:quaternion-algebras-Hida-families}, the residual representation $\bar{\T}=\T/\m_{\mathcal{R}}\T$ is equivalent up to a finite base change to the residual representation $\bar{\rho}_f$, which is absolutely irreducible by Assumption \ref{ass:residual-representation-irreducible}. We conclude by noting that tensoring with one-dimensional characters preserves irreducibility.
        \end{proof}

        As explained in \cite[(3)]{Howard07:variation-of-Heegner-points-in-Hida-families}, there is a perfect, alternating, $G_\Q$-invariant $\calR$-bilinear pairing
        \begin{equation}\label{eq:pairing}
            \T^\dag\times\T^\dag\longrightarrow\calR(1),
        \end{equation}
        where $\calR(1)$ is the usual Tate twist of $\calR$. The existence of this pairing implies that the representation $\T^\dag$ (and its quotients) decomposes into the sum of two $\calR$-submodules of rank 1, associated with the eigenvalues $1$ and $-1$ for the action of the complex conjugation $\tau_c$.

        As noted in \cite[p.~300]{Longo-Vigni11:quaternion-algebras-Hida-families}, a consequence of Assumption \ref{ass:residual-representation-irreducible} is that, if $\mathfrak{D}_p$ is a decomposition group for $p$ in $G_{\Q}$, there is an exact sequence of $\mathcal{R}\llbracket \mathfrak{D}_p\rrbracket$-modules
        \begin{equation}\label{eq:exact-sequence-at-p}
            0\longrightarrow F_p^+(\T^\dag)\longrightarrow \T^\dag\longrightarrow F_p^-(\T^\dag)\longrightarrow 0,
        \end{equation}
        where $F_p^+(\T^\dag)$ and $F_p^-(\T^\dag)$ are free $\mathcal{R}$-modules of rank $1$.

\subsection{Big Heegner classes} \label{sec:big-heegner-classes}

For every positive integer $c$ coprime with $N$, in \cite[Definition 7.4]{Longo-Vigni11:quaternion-algebras-Hida-families} Longo and Vigni defined the \emph{big Heegner class of conductor $c$}
\begin{equation*}
    \kappa_c\in H^1(K[c],\T^{\dag}),
\end{equation*}
where $K[c]$ is the ring class field of the imaginary quadratic field $K$ (fixed in Assumption \ref{ass:field-K}) of conductor $c$. If $L/F$ is a finite field extension, denote by $\Cor^L_F$ the corestriction map in cohomology. For the purpose of this work, we don't need to recall the definition of Longo--Vigni's big Heegner classes, since we will only rely on their Euler system properties, proved in \cite[§8]{Longo-Vigni11:quaternion-algebras-Hida-families}. 

\begin{proposition}\label{prop:properties-big-Heegner-points}
    Let $c\in\Z_{>0}$ coprime with $N$ and $\ell\nmid Npc$ a prime of $\Q$ inert in $K$. Then
    \begin{itemize}
        \item[(a)] $ U_p(\kappa_c)=\Cor_{K[c]}^{K[cp]}(\kappa_{cp})$;
        \item[(b)] $T_\ell(\kappa_c)=\Cor_{K[c]}^{K[c\ell]}(\kappa_{c\ell})$.
    \end{itemize}
\end{proposition}
\begin{proof}
    See \cite[Corollaries 8.2 and 8.4]{Longo-Vigni11:quaternion-algebras-Hida-families}.
\end{proof}

Keep the same notation of the previous proposition. By class field theory, the prime $\gl=(\ell)\subseteq K$ splits completely in the extension $K[c]/K$.  Fix a prime $\gl_c\in K[c]$ above $\gl$. Then, $\gl_c$ is totally ramified in $K[c\ell]$, so that $\gl_c\mathcal{O}_{K[c\ell]}=\gl_{c\ell}^{\ell+1}$ for a prime ideal $\gl_{c\ell}$ of $K[c\ell]$. Denote by $\Fr_{\gl_c/\ell}\in\Gal(K[c]/\Q)$ the arithmetic Frobenius for the extension $\gl_c/\ell$.

\begin{proposition}[Eichler--Shimura relation]\label{prop:Eichler-Shimura-relation-big-Heegner-points}
        Let $\ell\nmid Npc$ be a prime inert in $K$. Then $\kappa_{c\ell}$ and $\Fr_{\gl_c/\ell}(\kappa_c)$ have the same image in $H^1(K[c\ell]_{\gl_{c\ell}},\T^\dag)$.
\end{proposition}
\begin{proof}
    See \cite[Proposition 8.7]{Longo-Vigni11:quaternion-algebras-Hida-families}.
\end{proof}

\begin{remark}
    Propositions \ref{prop:properties-big-Heegner-points} and \ref{prop:Eichler-Shimura-relation-big-Heegner-points} imply that the classes $\kappa_c$ form a $p$-complete anticyclotomic Euler system, in the language of \cite{mastella-zerman2025:anticyclotomic}. Indeed, part of this paper can be seen as a motivating example of the theory developed therein. See \cite[§4.3]{mastella-zerman2025:anticyclotomic} for further links between these two articles.
\end{remark}

\subsection{Anticyclotomic twists}\label{sec:anticyclotomic-twists}

For every $t\in\Z_{>0}$, let $K_t$ be the maximal $p$-extension in $K[p^{t+1}]/K$ and set $K_0:=K$. By \cite[Theorem 7.24]{Cox22:primes-of-the-form} and Assumption \ref{ass:field-K}, we have that $[K_t:K]=p^t$. 

\begin{definition}
    The field $K_{\infty}:=\bigcup_{t\ge 0} K_t$ is called the \emph{anticyclotomic $\Z_p$-extension} of $K$.
\end{definition}

Let now $n\in\Z_{>0}$ coprime with $Np$ and set $K_t[n]$ to be the composite of $K_t$ and $K[n]$. The fields $K_t$ and $K[n]$ are disjoint over $K$ thanks to ramification issues and Assumption \ref{ass:field-K}. In particular, this implies that any prime $v$ of $K$ that lies above $p$ is totally ramified in $K_{t}$.

\begin{definition}
     Define $\Gamma^{\ac}:=\Gal(K_{\infty}/K)\cong \Z_p$ and $\gL^{\ac}:=\Z_p\llbracket \Gamma^{\ac}\rrbracket$. Moreover, fix a profinite generator $\gamma_{\ac}$ of $\Gamma^{\ac}$.
\end{definition}

\begin{definition}
    Call $\mathcal{R}^{\ac}=\mathcal{R}\llbracket \Gamma^{\ac}\rrbracket\cong\mathcal{R} \ \hat{\otimes}_{\Z_p} \gL^{\ac}$ and set $\T^{\ac}=\T^{\dag}\otimes_{\mathcal{R}}\mathcal{R}^{\ac}\cong\T^{\dag}\hat{\otimes}_{\Z_p}\gL^{\ac}$. Here, the symbol $\hat{\otimes}$ denotes the completed tensor product between topological $\Z_p$-modules.
\end{definition}

We make $\mathcal{R}^{\ac}$ a left $G_K=\Gal(\bar{\Q}/K)$-module via the natural projection $G_K\twoheadrightarrow \Gamma^{\ac}$. Similarly, $\T^{\ac}$ is a $G_K$-module via the action coming from both sides of the tensor product. For our intended applications in anticyclotomic Iwasawa theory, we will mainly work with the representation $\T^{\ac}$ and some of its quotients.

\begin{lemma}\label{lem:residual-representation-Iwasawa-irreducible}
    The $\mathcal{R}^{\ac}$-module $\T^{\ac}$ is free of rank $2$. As a $G_K$-representation, it is unramified outside $Np$. Its residual $G_K$-representation is isomorphic to $\bar{\T}^{\dag}$.
\end{lemma}
\begin{proof}
    The first two statements are a direct consequence of Proposition \ref{prop:characteristic-polynomial-Frobenius} and the fact that $\Gamma^{\ac}$ does not contain inertia outside $p$.

    If $\m^{\ac}$ denotes the maximal ideal of $\mathcal{R}^{\ac}$, there is a canonical isomorphism of $G_K$-modules
    \begin{equation*}
        \T^{\ac}/\m^{\ac}\T^{\ac}=\T^\dag/\m_{\mathcal{R}}\T^\dag\otimes_{\F_\mathcal{R}}\F_{\mathcal{R}^{\ac}}
    \end{equation*}
    where $\F_\mathcal{R}$ is the residue field of $\mathcal{R}$ and $\F_{\mathcal{R}^{\ac}}$ is the residue field of $\mathcal{R}^{\ac}$. Since $[\gamma_{\ac}]-1\in\m^{\ac}$, the action of $G_K$ on $\F_{\mathcal{R}^{\ac}}$ is trivial. Moreover, the isomorphism $\mathcal{R}^{\ac}\cong \mathcal{R}\llbracket X\rrbracket$ as $\mathcal{R}$-algebras (see \cite[Proposition 5.3.5]{Neukirch-Schmidt-Wingberg:cohomology-of-number-fields}) implies that $\F_{\mathcal{R}^{\ac}}\cong \F_\mathcal{R}$ as fields. 
\end{proof}

\subsection{Quotients, Galois groups and primes}\label{sec:quotients-galois-groups-and-primes}

Inspired by the setup of \cite{buyukboduk2011:lambda, Buyukboduk:big-Heegner-point-kolyvagin-system, Buyukboduk:deformations-of-kolyvagin-systems}, we define some relevant classes of quotients of $\mathcal{R}$, $\mathcal{R}^\ac$, $\T^\dag$ and $\T^\ac$. For every $m,s\in\Z_{>0}$ and $t\ge 0$, define $\calR_m:=\mathcal{R}/(\omega_{2,m})$, $\calR_{m,s}:=\mathcal{R}/(\omega_{2,m},p^s)$ and $\calR_{t,m,s}:=\mathcal{R}^{\ac}/(\omega_{2,m}, p^s, \gamma_{\ac}^{p^t}-1)$. Define also 
    \begin{equation*}
        \T_m:=\T^\dag\otimes_{\mathcal{R}}\calR_m,\quad \T_{m,s}:=\T^\dag\otimes_{\mathcal{R}}\calR_{m,s}\quad\text{and}\quad \T_{t,m,s}:=\T^{\ac}\otimes_{\mathcal{R}^{\ac}}\calR_{t,m,s}.
    \end{equation*}

\begin{lemma}
    Let $m,s\in\Z_{>0}$ and $t\ge 0$. Then
    \begin{itemize}
        \item[(a)] $\omega_{2,m}\in\m_\calR^{m}$.
        \item[(b)] $\calR_m$ is finite and free over $\OF$.
        \item[(c)] $\calR_{m,s}$ and $\calR_{t,m,s}$ are finite of $p$-power order.
        \item[(d)] $\calR=\varprojlim_{m,s}\calR_{m,s}$ and $\calR^{\ac}=\varprojlim_{t,m,s}\calR_{t,m,s}$.
    \end{itemize}
\end{lemma}
\begin{proof}
    (a) Recall that $\omega_{2,m}=[\gamma]^{p^{m-1}}-\gamma^{p^{m-1}}$, where $\gamma$ is a profinite generator of $\Gamma=1+p\Z_p$. This implies that $\gamma^{p^{m-1}}\in \Gamma^{p^{m-1}}=1+p^m\Z_p$. A classical computation (see \cite[p.116]{Washington:cyclotomic-fields}) yields also that $[\gamma]^{p^{m-1}}-1\in \m_{\Lambda_F}^{m}$, therefore $\omega_{2,m}\in\m_\calR^m$. 

    (b) By Lemma \ref{lem:algebraic-properties-of-R}, we know that $\calR$ is finite and free over $\gL_F$, therefore $\calR_m$ is finite and free over $\gL_F/(\omega_{2,m})$. Since $\omega_{2,m}$ corresponds to a distinguished polynomial of $\OF\llbracket X\rrbracket$ (see \cite[p.116]{Washington:cyclotomic-fields}), the division by $\omega_{2,m}$ induces an isomorphism of $\gL_F/(\omega_{2,m})$ with the elements of $\OF[X]$ of degree $<n$ (see e.g.~\cite[Proposition 7.2]{Washington:cyclotomic-fields}), which is a finitely generated free $\OF$-module.

    (c) Clear from point (b).

    (d) We prove the equality only for $\calR^{\ac}$, since the other one is dealt in the exact same way. Let $\m_{t,m,s}$ be the maximal ideal of $\calR_{t,m,s}$. By Nakayama's lemma, there is an exponent $e_{t,m,s}\in\Z_{>0}$ such that $\m_{t,m,s}^{e_{t,m,s}}=\{0\}$. Therefore, by the proof of point (a), we have that
    \begin{equation*}
        \m_{\calR^\ac}^{e_{t,m,s}}\subseteq (\omega_{2,m}, p^s, \gamma_{\ac}^{p^t}-1)\subseteq \m_{\calR^\ac}^{f_{t,m,s}},
    \end{equation*}
    where $f_{t,m,s}:=\min\{t,m,s\}$. Therefore, the topology on $\calR^\ac$ induced by the family of ideals $(\omega_{2,m}, p^s, \gamma_{\ac}^{p^t}-1)$ coincides with the $\m_{\calR^\ac}$-topology, giving $\calR^{\ac}=\varprojlim_{t,m,s}\calR_{t,m,s}$.
\end{proof}

    \begin{definition}\label{dfn:admissible-primes}
    Let $m,s\in\Z_{>0}$.
	\begin{itemize}
			\item[(a)] Define $\mathcal{P}$ to be the set of all primes $\ell$ inert in $K$ such that $\ell\nmid Np$.
			\item[(b)] Define $\mathcal{P}_{m,s}$ to be the subset of $\mathcal{P}$ made by all primes $\ell$ such that the arithmetic Frobenius $\Fr_\ell$ is conjugated with the complex conjugation $\tau_c$ in $\Gal(K(\T_{m,s})/\Q)$.
			\item[(c)] Define the sets $\mathcal{N}$ and $\mathcal{N}_{m,s}$ to be the sets of all square-free products of elements of $\mathcal{P}$ and $\mathcal{P}_{m,s}$, respectively.
                \item[(d)] For $n \in\mathcal{N}$, define $\mathcal{G}_n:=\Gal(K[n]/K[1])$.
	\end{itemize}
    \end{definition}

    By the Chebotarev density theorem, each set $\mathcal{P}_{m,s}$ consists of infinitely many primes and $\mathcal{P}_{m,s}\subseteq\mathcal{P}_{m',s'}$ whenever $m\ge m'$ and $s\ge s'$. For every $n\in\mathcal{N}$ we also have that $\mathcal{G}_n\cong \prod_{\ell\mid n}\mathcal{G}_\ell$, where the product varies among all prime divisors of $n$, and each group $\mathcal{G}_\ell$ is cyclic of order $\ell+1$ (see e.g.~\cite[§3]{Gross91:Kolyvagins-work}).

    \begin{lemma}\label{lem:Tl-acts-as-0}
        Let $\ell\in\mathcal{P}_{m,s}$ and call $\gl$ the prime of $K$ above $\ell$. Then
        \begin{itemize}
            \item[(i)] The arithmetic Frobenius $\Fr_\lambda$ acts trivially on $\T_{t,m,s}$, for every $t\in\Z_{>0}$. 
            \item[(ii)] $\ell\equiv -1\bmod p^s$ and the image of the Hecke operator $T_\ell$ in $\calR_m$ is divisible by $p^s$.
        \end{itemize}
    \end{lemma}
    \begin{proof}
        (i) The element $\Fr_\gl=\Fr_\ell^2$ acts trivially on $\T_{m,s}$ by condition (b) of Definition \ref{dfn:admissible-primes} and it acts trivially on $\gL^{\ac}$ since $\gl$ is split in $K_\infty/K$ by class field theory.

        (ii) As explained after Equation \eqref{eq:pairing} (see also \cite[Lemma 4.1.15]{Zerman24:PhD-thesis}), the characteristic polynomial for the action of $\tau_c$ on $\T_m$ is $X^2-1$. By comparing this polynomial with the characteristic polynomial of $\Fr_\ell$ computed in Proposition \ref{prop:characteristic-polynomial-Frobenius}, we obtain the claimed relations.
    \end{proof}

\begin{lemma}
 \label{lem:charachteristic-polynomial-frobenius-Tm}
     Let $s\ge m$ and $\ell\in\mathcal{P}_{m,s}$. Then, $\Fr_\ell$ acts on $\T_m$ with characteristic polynomial
     \begin{equation*}
         X^2-(-1)^{\frac{k+j}{2}-1} T_\ell X+\ell.
     \end{equation*}
 \end{lemma}
 \begin{proof}
     In light of Proposition \ref{prop:characteristic-polynomial-Frobenius}, it is enough to prove that $\Theta(\Fr_\ell)=(-1)^{\frac{k+j}{2}-1}$ in $\calR_m$. Since $\ell\equiv -1\bmod{p^s}$ and $\e_{\cyc}(\Fr_\ell)=\ell$, we obtain that
     \begin{equation*}
         \Theta(\Fr_\ell)=(-1)^{\frac{k+j}{2}-1}[x^{\frac{1}{2}}]\in\gL_F^\times
     \end{equation*}
     for some $x\in1+p^s\Z_p$, where $x^{\frac{1}{2}}$ is the unique square root of $x$ in $1+p\Z_p$. Since $p>2$, the element $x^{\frac{1}{2}}$ lies in $1+p^s\Z_p=\Gamma^{p^{s-1}}$. Since $s\ge m$, by Theorem \ref{thm:first-Hida-structure-theorem} we know that the image of $\Gamma^{p^{s-1}}$ in $\calR_m$ is trivial.
 \end{proof}

\subsection{Galois invariants over abelian extensions of $K$}

The following lemma is a variant of Nakayama's lemma.

\begin{lemma}\label{lem:nakayama-for-Galois-representation}
    Let $A$ be a local Noetherian ring with maximal ideal $\m_A$ and $T$ be a finitely generated $A$-torsion-free module. Let also $G$ be a topological group that acts continuously and $A$-linearly on $T$. If $H^0(G,T/\mathfrak{m}_A T)=\{0\}$, then $H^0(G,T)=\{0\}$.
\end{lemma}
\begin{proof}
    The submodule $(T^G+\mathfrak{m}_A T)/\mathfrak{m}_A T$ is contained in $(T/\mathfrak{m}_A T)^G$, hence it is trivial. This implies that $T^G\subseteq \mathfrak{m}_A T$. Then, we can write any nonzero element $x\in T^G$ as $x=ay$ with $a\in\mathfrak{m}_A\setminus\{0\}$ and $y\in T$. For every $\gs\in G$, one has the relation
    \begin{equation*}
        0=\gs(ay)-ay=a(\gs(y)-y).
    \end{equation*}
    Since $T$ is $A$-torsion-free, we obtain that $\gs(y)-y=0$, i.e.~$y\in T^G$. Hence, $T^G\subseteq \mathfrak{m}_A T^G$. Since $T$ is finitely generated over the Noetherian ring $A$, then also $T^G$ is finitely generated over $A$, and we conclude by applying Nakayama's lemma.
\end{proof}

We now apply this criterion to show that some relevant representations have no invariants over the absolute Galois group of some ring class fields. Once again, fix $m,s\in\Z_{>0}$, $t\ge 0$ and call $D_K$ the discriminant of $K$.

\begin{lemma}\label{lem:no-invariants}
    Let $n$ be a positive integer coprime with $NpD_K$ and let $T$ be any of $\T^{\dag}$, $\T^{\ac}$, $\T_m$, $\T_{m,s}$ and $\T_{t,m,s}$. Then $H^0(K[n],T)=\{0\}$.
\end{lemma}
\begin{proof}
    Let $T$ be one of $\T^{\dag}$,  $\T_m$ or $\T_{m,s}$ and denote by $\bar{T}$ the residual $G_\Q$-representation of $T$, which is a free module of rank $2$ over a finite field, unramified outside $Np$. Let $F:=\Q(\bar{T})\cap K[n]$ and let $q$ be a prime that ramifies in $F$. Since the extension $\Q(\bar{T})/\Q$ is unramified outside $Np$, we must have that $q\mid Np$. From the fact that the extension $K[n]/\Q$ is unramified outside $nD_K$, we must have that $q\mid nD_K$. Since $Np$ is coprime with $n D_K$ (see Assumption \ref{ass:field-K}), no prime of $\Q$ ramifies in $F$. Therefore, $F=\Q$. Hence, by restricting automorphisms from $\bar{\Q}$ to $\Q(\bar{T})$ we obtain a surjection $G_{K[n]}\twoheadrightarrow\Gal(\Q(\bar{T})/\Q)$. Since, by Lemma \ref{lem:twist-Hida-representation-irreducible}, the representation $\bar{T}$ has no $\Gal(\Q(\bar{T})/\Q)$-invariants, we conclude that it does not have any nontrivial $G_{K[n]}$-invariant. By Lemma \ref{lem:nakayama-for-Galois-representation}, it follows that $H^0(K[n],T)=\{0\}$.

    By Lemma \ref{lem:residual-representation-Iwasawa-irreducible} there is an isomorphism of $G_K$-representations between $\bar{\T}^{\dag}$ and $\bar{\T}^{\ac}$, therefore $H^0(K[n],\bar{\T}^{\ac})=\{0\}$. By applying Lemma \ref{lem:nakayama-for-Galois-representation} again, we obtain that $\T^{\ac}$ and $\T_{t,m,s}$ have no $G_{K[n]}$-invariants.
\end{proof}

\begin{remark}\label{rk:irreducibility-over-ring-class-fields}
    In the first part of the proof of Lemma \ref{lem:no-invariants} we showed that, under our running assumptions, the residual representation $\bar{\T}^\dag=\bar{\T}^\ac$ is $G_{K[n]}$-irreducible, for every $n$ coprime with $NpD_K$.
\end{remark}

\subsection{Shapiro's lemma}\label{sec:shapiro-lemma}

Let $F$ be a perfect field and $F_\infty/F$ be a $\Z_p$-extension. For any $t\ge 0$, call $F_t$ the $t$-th layer of the extension. Let $T$ be a $\Z_p\llbracket G_F\rrbracket$-module. We allow $G_F$ to act on both factors of $T\otimes_{\Z_p}\Z_p[\Gal(F_t/F)]$ and $T \ \hat{\otimes}_{\Z_p}\Z_p\llbracket  \Gal(F_\infty/F)\rrbracket$ via the natural projections.

    \begin{lemma}\label{lem:Shapiro-application}
        Let $t\ge 0$. Shapiro's lemma induces isomorphisms
        \begin{itemize}
            \item[(i)] $\sh_t\colon H^1(F_t,T)\cong H^1\big(F, T\otimes_{\Z_p}\Z_p[\Gal(F_t/F)]\big)$;
            \item[(ii)] $\sh_\infty\colon \varprojlim_{t} H^1(F_t,T)\cong H^1\big(F, T \ \hat{\otimes}_{\Z_p}\Z_p\llbracket \Gal(F_\infty/F)\rrbracket\big)$, where the inverse limit is taken with respect to the corestriction maps.\label{condition:shapiro-anticyclotomic-p-extension}
        \end{itemize}
    \end{lemma}
    \begin{proof}
        See \cite[ Proposition II.1.1]{Colmez:theorie-dIwasawa}. 
    \end{proof}

      Let now $\ga>t\ge 1$ and call $\psi_{\ga,t}:\Z_p[\Gal(F_\ga,F)]\to \Z_p[\Gal(F_t,F)]$ the map induced by the natural projection between Galois groups. In the proof of \cite[Proposition II.1.1]{Colmez:theorie-dIwasawa}, point (ii) of the previous lemma is shown using the commutativity of
    \begin{equation}\label{eq:diagram-Shapiro-lemma}
        \begin{tikzcd}
	{H^1(F_\ga,T)} & {H^1\big(F, T\otimes_{\Z_p}\Z_p[\Gal(F_\ga/F)]\big)} \\
	{H^1(F_t,T)} & {H^1\big(F, T\otimes_{\Z_p}\Z_p[\Gal(F_t/F)]\big),}
	\arrow["\sh_\ga", from=1-1, to=1-2]
	\arrow["\sh_t", from=2-1, to=2-2]
	\arrow["\Cor"', from=1-1, to=2-1]
	\arrow["{\psi_{\ga,t}}", from=1-2, to=2-2]
\end{tikzcd}
    \end{equation}
    where the left vertical map is corestriction and the right one is the map induced by $\psi_{\ga,t}$.

    Coming back to our arithmetic setting, Lemma \ref{lem:Shapiro-application} yields isomorphisms
    \begin{equation}\label{eq:shapiro-isomorphisms}
        \sh_t\colon H^1(K_t,\T_{m,s})\cong H^1(K,\T_{t,m,s})\quad\text{and}\quad  \sh_\infty\colon \varprojlim_t H^1(K_t,\T^\dag)\cong H^1(K,\T^{\ac})
    \end{equation}
    for every $m,s\in\Z_{>0}$ and $t\ge 0$. Therefore, studying the cohomology of $\T^{\ac}$ over $K$ is equivalent to studying the cohomology of $\T^\dag$ over the anticyclotomic tower.

\section{Selmer structures and Kolyvagin systems}\label{sec:selmer-structures-and-kolyvagin-systems}

The aim of this section is to give the definition of \emph{modified universal Kolyvagin system} for the representation $\T^{\ac}$. This is the slight generalization of the classical notion of universal Kolyvagin system (see \cite{Mazur-Rubin:Kolyvagin-systems, Howard04:heegner-point-kolyvagin-system, Buyukboduk:deformations-of-kolyvagin-systems}) used and studied in \cite{mastella-zerman2025:anticyclotomic}.

\subsection{Local conditions and Selmer structures}

Let $L$ be a finite extension of $K$, let $R$ be a complete local Noetherian ring with finite residue field of characteristic $p$ and let $M$ be an $R\llbracket G_L\rrbracket$-module which is finitely generated and free as an $R$-module and unramified outside $Np$.

\begin{definition}
    A \textit{local condition} $\mathcal{F}$ on $M$ at a place $w$ of $L$
is the choice of an $R$-submodule $H_{\mathcal{F}}^1(L_w,M)\subseteq H^1(L_w,M)$. 
\end{definition}

\begin{remark}
    If $w$ is an archimedean prime of $L$, then $L_w=\C$ and so $H^1(L_w,M)=\{0\}$. Therefore, the choice of local conditions at archimedean primes is uniquely determined and will not play a role.
\end{remark}

 \begin{definition}\label{dfn:propagated-condition}
     Given an $R\llbracket G_L\rrbracket$-submodule (resp.~quotient) $N$ of $M$ and a local condition $\mathcal{F}$ on $M$, the \emph{propagated condition} on $N$ is the preimage (resp.~image) of $H_{\mathcal{F}}^1(L_w,M)$ under the natural map $H^1(L_w,N)\to H^1(L_w,M)$ (resp.~$ H^1(L_w,M)\to H^1(L_w,N)$).
 \end{definition}

 \begin{definition}
     A \emph{Selmer structure $\mathcal{F}$} on $M$ is a collection of local conditions on $M$ at every prime $w$ of $L$. The \emph{Selmer module} attached to $\mathcal{F}$ is
     \begin{equation*}
            H^1_{\mathcal{F}}(L,M):=\ker\bigg(H^1(L,M)\to\prod_w \frac{H^1(L_w,M)}{H^1_{\mathcal{F}}(L_w,M)}\bigg),
        \end{equation*}
        where $w$ runs over all places of $L$.
 \end{definition}

 We will be primarily concerned with local conditions of the following type.
 \begin{itemize}
     \item[(i)] The \emph{unramified} local condition on $M$ at $w$ is
     \begin{equation*}
         H^1_{\ur}(L_w,M):=\ker\big(H^1(L_w,M)\xrightarrow{\res} H^1(L_w^{\ur},M)\big),
     \end{equation*}
     where $L_w^{\ur}$ is the maximal unramified extension of $L_w$. 
     \item[(ii)] If  $L_w$ has residue characteristic different from $p$, the \emph{$F$-transverse} local condition on $M$ at $v$ is
     \begin{equation*}
         H^1_{F-\tr}(L_w,M):=\ker\big(H^1(L_w,M)\xrightarrow{\res} H^1(F,M)\big),
     \end{equation*}
     where $F$ is a maximal totally tamely ramified abelian $p$-extension of $L_w$.
\end{itemize}

Let $K(\ell)$ be the maximal $p$-subextension of $K[\ell]/K$. When $L=K$ and $w=\gl$ is a prime of $K$ that lies above an element of $\mathcal{P}_{m,s}$ for some $m,s\in\Z_{>0}$, local class field theory implies that $F=K(\ell)_{\lambda_\ell}$ is a maximal totally tamely ramified abelian $p$-extension of $K_\gl$, where $\gl_\ell$ is the unique prime of $K(\ell)$ above $\gl$ (see \cite[Lemma 4.1.21]{Zerman24:PhD-thesis}). For simplicity, we set $H^1_{\tr}(K_\gl,M):=H^1_{F-\tr}(K_\gl,M)$ and call it the \emph{transverse condition}.
    
    \begin{remark}
    \label{rk:transverse-condition}    
        If $M$ is unramified at $\gl$, using the fact that $M$ is a $\Z_p$-module and that $K(\ell)_{\lambda_\ell}$ is the maximal $p$-extension of $K_\gl$ contained in $K[\ell]_{\gl_\ell}$, one can also easily prove that
        \begin{equation*}
         H^1_{\tr}(K_\gl,M)=\ker\big(H^1(K_\gl,M)\xrightarrow{\res} H^1(K[\ell]_{\gl_\ell'},M)\big)
     \end{equation*}
     for every prime $\gl_\ell'$ of $K[\ell]$ above $\gl_\ell$. For more details, see \cite[Lemma 4.2.8]{Zerman24:PhD-thesis}.
    \end{remark}

    \begin{definition}
         When $L_w$ has residue characteristic different from $p$ and $M$ is unramified at $w$, we define the \emph{singular quotient} 
         \begin{equation*}
             H_{\s}^1(L_w,M):=\frac{H^1(L_w,M)}{H_{\ur}^1(L_w,M)}.
         \end{equation*}
    \end{definition}

\subsection{Greenberg conditions}

Let $v$ be a place of $K$ above $p$ and fix an absolute decomposition group $\mathfrak{D}_v$ for $v$. Since $\mathcal{R}^{\ac}$ is flat over $\mathcal{R}$, by tensoring the sequence in \eqref{eq:exact-sequence-at-p} with $\mathcal{R}^{\ac}$ we obtain an exact sequence of $\mathcal{R}^\ac\llbracket \mathfrak{D}_v\rrbracket$-modules
\begin{equation}\label{eq:exact-sequence-at-p-Iwasawa}
            0\longrightarrow F_v^+(\T^{\ac})\longrightarrow \T^{\ac}\longrightarrow F_v^-(\T^{\ac})\longrightarrow 0
\end{equation}
 where $F_v^+(\T^{\ac})$ and $F_v^-(\T^{\ac})$ are free $\mathcal{R}^{\ac}$-modules of rank $1$.

 \begin{definition}\label{dfn:strict-Greenberg-selmer-structure}
        Let $L$ be a finite extension of $K$. The \emph{(strict) Greenberg Selmer structure} $\mathcal{F}_{\Gr}$ on $\T^{\ac}$ over $L$ is defined by setting 
        \begin{equation*}
            H^1_{\mathcal{F}_{\Gr}}(L_w,\T^{\ac}):=\begin{cases}
                H_{\ur}^1(L_w,\T^{\ac})\quad &\text{if $w\nmid p$}\\
                \ker\big(H^1(L_w,\T^{\ac})\to H^1(L_w,F_v^-(\T^{\ac}))\big)\quad &\text{if $w\mid p$}
            \end{cases}
        \end{equation*}
        where $w$ runs over all places of $L$ and the unnamed map is induced by \eqref{eq:exact-sequence-at-p-Iwasawa}.
    \end{definition}

    In the same way, one defines the (strict) Greenberg Selmer structure for the representation $\T^{\dag}$, just using the exact sequence  \eqref{eq:exact-sequence-at-p} in place of \eqref{eq:exact-sequence-at-p-Iwasawa}. Moreover, the Selmer structure $\mathcal{F}_{\Gr}$ propagates (in the sense of Definition \ref{dfn:propagated-condition}) to the quotients $\T_m$, $\T_{m,s}$ and $\T_{t,m,s}$ for every $m,s\in\Z_{>0}$ and $t\ge 0$.

\subsection{Modified Kolyvagin systems}\label{sec:modified-Kolyvagin-systems}

Fix now $m,s\in\Z_{>0}$, $t\ge 0$, $\ell\in\mathcal{P}_{m,s}$ and call $\gl=(\ell)$ the prime of $K$ above $\ell$. There is a chain of functorial isomorphisms
\begin{equation}\label{eq:definition-finite-singular}
    \phi_\lambda^{\fs}\colon H^1_{\ur}(K_\gl,\T_{t,m,s})\overset{\ga_\gl}{\longrightarrow} \T_{t,m,s}\overset{\gb_\gl^{-1}}{\longrightarrow} H^1_{\s}(K_\gl,\T_{t,m,s})\otimes \mathcal{G}_\ell
\end{equation}
that come from \cite[Proposition 4.7]{Buyukboduk:big-Heegner-point-kolyvagin-system}, which is called the \emph{finite-singular isomorphism}.

\begin{remark}
        As noted in the proof of \cite[Proposition 1.1.7]{Howard04:heegner-point-kolyvagin-system}, the isomorphism $\ga_\gl$ is given by evaluating cocycles at the Frobenius automorphism $\Fr_\gl$, whereas $\gb_\gl$ is given by sending the class of $\xi\otimes\gs$ to $\xi(\tilde{\gs})$, for any lift $\tilde{\gs}\in\Gal(\bar{\Q}_\ell/K_\gl^{\ur})$ of $\gs$. For more details, see also \cite[Lemma 2.8]{mastella-zerman2025:anticyclotomic}.
\end{remark}

For every $\chi\in \Aut_{\calR_{t,m,s}}(\T_{t,m,s})$, we define
\begin{equation*}
    \phi_\lambda^{\fs}(\chi)\colon H^1_{\ur}(K_\gl,\T_{t,m,s})\overset{\ga_\gl}{\longrightarrow} \T_{t,m,s}\overset{\chi}{\longrightarrow} \T_{t,m,s}\overset{\gb_\gl^{-1}}{\longrightarrow} H^1_{\s}(K_\gl,\T_{t,m,s})\otimes \mathcal{G}_\ell.
\end{equation*}
By \cite[Lemma 4.10]{Buyukboduk:big-Heegner-point-kolyvagin-system} (see also \cite[Lemma 1.2.4]{Mazur-Rubin:Kolyvagin-systems}), there is a functorial splitting 
\begin{equation*}
            H^1(K_\gl,\T_{t,m,s})=H_{\ur}^1(K_\gl,\T_{t,m,s})\oplus H_{\tr}^1(K_\gl,\T_{t,m,s}),
\end{equation*}
so that $H_{\tr}^1(K_\gl,\T_{t,m,s})$ projects isomorphically onto $H_{\s}^1(K_\gl,\T_{t,m,s})$. We now use the transverse condition to modify existing Selmer structures.

\begin{definition}
        Let $\mathcal{F}$ be a Selmer structure on $\T_{t,m,s}$ over $K$. For every $n\in\mathcal{N}_{m,s}$, the \emph{modified Selmer structure} $\mathcal{F}(n)$ on $\T_{t,m,s}$ is given by
        \begin{equation*}
            H^1_{\mathcal{F}(n)}(K_v,\T_{t,m,s}):=\begin{cases}
                H^1_{\mathcal{F}}(K_v,\T_{t,m,s}) &\text{if $v\nmid n$}\\
                H^1_{\tr}(K_v,\T_{t,m,s}) &\text{if $v\mid n$}
            \end{cases}
        \end{equation*}
        for every place $v$ of $K$.
\end{definition}

\begin{definition}
    For every $n\in\mathcal{N}$ set $\mathcal{G}(n)=\bigotimes_{\ell\mid n}\mathcal{G}_\ell$, where the tensor product runs over all primes dividing $n$.
\end{definition}
 
Let now $\mathcal{F}$ be a Selmer structure on $\T_{t,m,s}$ over $K$ with the property that $H^1_{\mathcal{F}}(K_v,\T_{t,m,s})=H^1_{\ur}(K_v,\T_{t,m,s})$ for every $v\nmid Np$. Following \cite[§3.1]{Buyukboduk:deformations-of-kolyvagin-systems}, for every $n\in\mathcal{N}_{m,s}$ and $\ell\in\mathcal{P}_{m,s}$ coprime with $n$, we define the maps:
 \begin{itemize}
            \item $\psi_{n\ell}^\ell\colon H^1_{\mathcal{F}(n\ell)}(K,\T_{t,m,s})\otimes\mathcal{G}(n\ell)\to H_{\s}^1(K_{\gl},\T_{t,m,s})\otimes\mathcal{G}(n\ell)$ to be localization at $\gl=(\ell)$ followed by the projection to the singular quotient;
            \item $\psi_n^\ell(\chi)\colon H^1_{\mathcal{F}(n)}(K,\T_{t,m,s})\otimes\mathcal{G}(n)\to H_{\s}^1(K_{\gl},\T_{t,m,s})\otimes\mathcal{G}(n\ell)$ to be localization at $\gl=(\ell)$ followed by the twisted finite-singular isomorphism $\phi_\gl^{\fs}(\chi)$.
\end{itemize}

Let now $\mathcal{P}_{m,s}'$ be an infinite subset of $\mathcal{P}_{m,s}$ and define $\mathcal{N}_{m,s}'$ to be the set of all square-free products of primes that lie below elements of $\mathcal{P}_{m,s}'$. Moreover, for every couple $(n,\ell)$ with $\ell$ prime and $n\ell\in\mathcal{N}_{m,s}'$, let $\chi_{n,\ell}$ be an automorphism of $\T_{t,m,s}$.

\begin{definition}\label{dfn:kolyvagin-system}
        A \emph{Kolyvagin system} for the quadruple $(\T_{t,m,s},\mathcal{F},\mathcal{P}_{m,s}', \{\chi_{n,\ell}\})$ is a collection of cohomology classes $\{\boldsymbol{\gk}^{\ac}(n)_{t,m,s}\}_{n\in\mathcal{N}_{m,s}'}$ such that, for every $n\ell\in\mathcal{N}_{m,s}'$ with $\ell$ prime, we have:
    \begin{itemize}
        \item[(K1)] $\boldsymbol{\gk}^{\ac}(n)_{t,m,s}\in H^1_{\mathcal{F}(n)}(K,\T_{t,m,s})\otimes\mathcal{G}(n)$;
        \item[(K2)] $\psi_n^\ell(\chi_{n,\ell})(\boldsymbol{\gk}^{\ac}(n)_{t,m,s})=\psi_{n\ell}^\ell(\boldsymbol{\gk}^{\ac}(n\ell)_{t,m,s})$.
    \end{itemize}
    We denote by $\KS(\T_{t,m,s},\mathcal{F},\mathcal{P}_{m,s}',\{\chi_{n,\ell}\})$ the $\calR_{t,m,s}$-module of all Kolyvagin systems for the quadruple $(\T_{t,m,s},\mathcal{F},\mathcal{P}_{m,s}',\{\chi_{n,\ell}\})$.
\end{definition}

Let now $\mathcal{P}'\subseteq \mathcal{P}_{1,1}$ such that $\mathcal{P}_{m,s}':=\mathcal{P}'\cap\mathcal{P}_{m,s}$ is an infinite subset of $\mathcal{P}_{m,s}$ for any $m,s\in\Z_{>0}$, and denote by $\mathcal{N}'$ the set of square-free products of elements of $\mathcal{P}'$. For every couple $(n,\ell)$ with $\ell$ prime and $n\ell\in\mathcal{N}'$, let $\chi_{n,\ell}\in\Aut_{\calR^{\ac}}(\T^{\ac})$ and denote with the same letter the induced automorphism of $\T_{t,m,s}$, for every $t\ge 0$ and $m,s\in\Z_{>0}$. Let $\mathcal{F}$ be a Selmer structure on $\T^\ac$ over $K$ with the property that $H^1_{\mathcal{F}}(K_v,\T^\ac)=H^1_{\ur}(K_v,\T^\ac)$ for every $v\nmid Np$. 

\begin{definition}
    The $\mathcal{R}^{\ac}$-module of \emph{universal Kolyvagin systems} for $(\T^{\ac},\mathcal{F},\mathcal{P}', \{\chi_{n,\ell}\})$ is
    \begin{equation*}
            \overline{\KS}(\T^{\ac},\mathcal{F},\mathcal{P}', \{\chi_{n,\ell}\}):=\varprojlim_{t,m,s}\varinjlim_{m',s'} \KS(\T_{t,m,s},\mathcal{F},\mathcal{P}_{m',s'}', \{\chi_{n,\ell}\}),
        \end{equation*}
    where the direct limit is taken with respect to all $m',s'$ with $m'\ge m$ and $s'\ge s$.
\end{definition}

\begin{remark}
    The attentive reader will notice that in the above definition we are using the fact that the propagation of $\mathcal{F}$ to $\T_{t,m,s}$ coincides with the unramified local condition at any prime $v\nmid Np$. For an explanation of this, see the first part of the proof of Lemma \ref{lem:easier-description-Greenberg-condition}.
\end{remark}

Assume now that $\boldsymbol{\gk}^{\ac}\in \overline{\KS}(\T^{\ac},\mathcal{F}_{\Gr},\mathcal{P}',\{\chi_{n,\ell}\})$ is induced by the set of modified Kolyvagin systems $\boldsymbol{\gk}^{\ac}(n)_{t,m,s}\in \KS(\T_{t,m,s},\mathcal{F}_{\Gr},\mathcal{P}_{m,s}',\{\chi_{n,\ell}\})$. Since $1\in\mathcal{N}_{m,s}'$ for every $m,s\in\Z_{>0}$, one can define
\begin{equation*}
    \boldsymbol{\gk}^{\ac}(1):=\varprojlim_{t,m,s} \boldsymbol{\gk}^{\ac}(1)_{t,m,s}\in \varprojlim_{t,m,s} H^1(K,T_{t,m,s})\cong H^1(K,\T^{\ac})\cong\varprojlim_t H^1(K_t,\T^{\dag}),
\end{equation*}
where the last isomorphism comes from \eqref{eq:shapiro-isomorphisms}.

\begin{remark}
    The aim of the next chapter will be to build a universal modified Kolyvagin system out of the big Heegner classes $\kappa_c$. In Lemma \ref{lem:key-formula}, we explicitly build the (inverse of) the twisting automorphisms $\chi_{n,\ell}$ that make the machinery work. We don't expect to be able to use these automorphisms to modify the definition of the derived classes in order to obtain a classical universal Kolyvagin system, as it is sometimes suggested in literature (see e.g. \cite[proof of Theorem 1.7.5]{Howard04:heegner-point-kolyvagin-system}). Indeed, one of the main obstacles is that the twisting automorphisms may not be Galois equivariant, and so they may not induce an automorphism of $H^1(K,\T_{t,m,s})$.
    
    For a broader discussion on \emph{modified} universal Kolyvagin systems, see \cite[§2.7]{mastella-zerman2025:anticyclotomic}. In any case, the slogan that one has to keep in mind is that, when it comes to bound the rank of Selmer groups, a modified universal Kolyvagin system has the same usage of a classical (universal) Kolyvagin system. We will make this clear and precise in Section \ref{sec:Anticyclotomic-Iwasawa-theory}.
\end{remark}

\section{The big Heegner point Kolyvagin system}\label{sec:the-big-heegner-point-kolyvagin-system}

In this chapter we show how to choose a suitable subset $\mathcal{P}'$ of $\mathcal{P}_{1,1}$ and build a modified universal Kolyvagin system for the triple $(\T^{\ac},\mathcal{F}_{\Gr},\mathcal{P}')$, starting from the set of Longo--Vigni's big Heegner classes recalled in Section \ref{sec:big-heegner-classes}. Before starting the actual construction, we pursue a deeper study of the ramification of  $\T^{\ac}$ at primes dividing $Np$.

\subsection{Controlling ramification}\label{sec:controlling-ramification}

Let $v$ be a prime of $K$ dividing $N$ and let $\mathfrak{I}_v$ be a fixed inertia group at $v$. For the rest of this article, we assume the following condition.

\begin{assumption}\label{ass:tamagawa-factors-at-N}
    For every $v\mid N$, the $\mathcal{R}$-module $H^1(\mathfrak{I}_v,\T^{\dag})$ is free.
\end{assumption}

\begin{remark}
    This condition has been widely studied in \cite[§4.3]{Buyukboduk-Sakamoto25:on-the-artin-formalism}. A simpler version of the arguments that go into the proof of \cite[Proposition 4.30]{Buyukboduk-Sakamoto25:on-the-artin-formalism} shows that, if the cardinality of the image of $\mathfrak{I}_v$ in $\Aut_{\mathcal{R}}(\T^\dag)$ is either $0$ or $\infty$ or it is not divisible by $p$, then $H^1(\mathfrak{I}_v,\T^{\dag})$ is free if the $p$-part of the Tamagawa factor for a single member of the Hida family passing through $f$ equals $1$.
\end{remark}

Since $\mathfrak{I}_v$ acts trivially on $\gL^{\ac}$ and the functor $-\ \hat{\otimes}_{{\Z}_p}\gL^{\ac}$ is exact, we have that $H^1(\mathfrak{I}_v,\T^{\ac})\cong H^1(\mathfrak{I}_v,\T^{\dag})\hat{\otimes}_{\Z_p}\gL^{\ac}$, therefore $H^1(\mathfrak{I}_v,\T^{\ac})$ is a free $\mathcal{R}^{\ac}$-module. Here we present an important consequence of Assumption \ref{ass:tamagawa-factors-at-N}.

\begin{proposition}\label{prop:tamagawa-numbers-at-N}
    A ring homomorphism $\mathfrak{s}\colon \mathcal{R}\to S$ induces an isomorphism $(\T^\dag)^{\mathfrak{I}_v}\otimes_{\mathcal{R}}S\cong (\T^{\dag}\otimes_{\mathcal{R}} S)^{\mathfrak{I}_v}$ and the module $H^1(\mathfrak{I}_v,\T^{\dag}\otimes_{\mathcal{R}}S)$ is $S$-torsion-free. The same holds replacing $\mathcal{R}$ with $\mathcal{R}^{\ac}$ and $\T^{\dag}$ with $\T^{\ac}$.
\end{proposition}
\begin{proof}
    Call $T_S=\T^{\dag}\otimes_{\mathcal{R}}S$. By \cite[Proposition 4.17]{Buyukboduk-Sakamoto25:on-the-artin-formalism} and Assumption \ref{ass:tamagawa-factors-at-N}, we obtain that $(\T^{\dag})^{\mathfrak{I}_v}\otimes_{\mathcal{R}}S\cong T_S^{\mathfrak{I}_v}$. Then, for every $x\in S\setminus\{0\}$, the map $T_S^{\mathfrak{I}_v}\to (T_S/xT_S)^{\mathfrak{I}_v}$ is surjective. Taking the long exact sequence in cohomology of the multiplication by $x$, we obtain that the multiplication by $x$ in $H^1(\mathfrak{I}_v,T_S)$ is injective. The same argument applies replacing $\mathcal{R}$ with $\mathcal{R}^{\ac}$ and $\T^{\dag}$ with $\T^{\ac}$.
\end{proof}

We now study the behaviour at $p$. Inspired by \cite[Hypothesis H.stz]{Buyukboduk:big-Heegner-point-kolyvagin-system}, for the rest of this article we work under the following assumption.

\begin{assumption}\label{ass:H.stz}
    For every prime $v\mid p$ of $K$ we assume that
    \begin{equation*}
        H^0(K_v, F_v^-(\bar{\T}^{\dag}))=\{0\},
    \end{equation*}
    where $F_v^-(\bar{\T}^{\dag}):=F_v^-(\T^{\dag})\otimes_{\mathcal{R}}\mathcal{R}/\mathfrak{m}_\mathcal{R}$.
\end{assumption}

 By Lemma \ref{lem:residual-representation-Iwasawa-irreducible}, the residual $G_K$-representation $\bar{\T}^\dag$ coincides with $\bar{\T}^{\ac}$. Therefore, Assumption \ref{ass:H.stz} is equivalent to assume that $H^0(K_v, F_v^-(\bar{\T}^{\ac}))=\{0\}$.

\begin{remark}
    Notice that $F_v^-(\bar{\T}^{\ac})\cong F_v^-(\bar{\T}^{\dag})$ is a vector space of dimension $1$ over a finite field. Moreover, the action of $G_{K_v}$ on it factors through the product of characters $\eta_v\Theta^{-1}$, where $\eta_v\colon G_{K_v}\to\mathcal{R}^{\times}$ is the unramified character defined by sending $\Fr_v$ to $U_p$ (see \cite[p. 300]{Longo-Vigni11:quaternion-algebras-Hida-families}). Therefore, Assumption \ref{ass:H.stz} is equivalent to requiring that the character $\eta_v\Theta^{-1}$ is not identically congruent to $1$ modulo $\m_{\mathcal{R}}$. 
\end{remark}

An important consequence of Assumptions \ref{ass:tamagawa-factors-at-N} and \ref{ass:H.stz} is an explicit description of the Greenberg local conditions on the quotients of $\T^{\ac}$. Most of the arguments appearing in the proof of the following result are present in \cite{Buyukboduk:big-Heegner-point-kolyvagin-system}, see in particular \cite[proof of Corollary 4.23]{Buyukboduk:big-Heegner-point-kolyvagin-system} and \cite[proof of Proposition 4.26]{Buyukboduk:big-Heegner-point-kolyvagin-system}.

\begin{lemma}\label{lem:easier-description-Greenberg-condition}
    Let $m,s\in\Z_{>0}$, $t\ge 0$ and $v$ be a place of $K$. Then
    \begin{equation*}
         H^1_{\mathcal{F}_{\Gr}}(K_v,\T_{t,m,s})=\begin{cases}
                H_{\ur}^1(K_v,\T_{t,m,s})\quad &\text{if $v\nmid p$}\\
                \ker\big(H^1(K_v,\T_{t,m,s})\to H^1(K_v,F_v^-(\T_{t,m,s}))\big)\quad &\text{if $v\mid p$},
            \end{cases}
    \end{equation*}
    where $F_v^-(\T_{t,m,s}):=F_v^-(\T^{\ac})\otimes_\mathcal{R^{\ac}}\calR_{t,m,s}$.
\end{lemma}
\begin{proof}
    Let $v\nmid Np$. The commutative diagram with exact rows
        \begin{equation*}
            \begin{tikzcd}
	{H^1_{\mathcal{F}_{\Gr}}(K_v,\T^{\ac})} & {H^1(K_v,\T^{\ac})} & {H^1(K_v^{\ur},\T^{\ac})} \\
	{H^1_{\ur}(K_v,\T_{t,m,s})} & {H^1(K_v,\T_{t,m,s})} & {H^1(K_v^{\ur},\T_{t,m,s})}
	\arrow[from=1-1, to=1-2]
	\arrow["\Res", from=1-2, to=1-3]
	\arrow["\Res", from=2-2, to=2-3]
	\arrow[from=2-1, to=2-2]
	\arrow[from=1-2, to=2-2]
	\arrow[from=1-3, to=2-3]
            \end{tikzcd}
        \end{equation*}
        induces a map $\phi\colon H^1_{\mathcal{F}_{\Gr}}(K_v,\T^{\ac})\longrightarrow H^1_{\ur}(K_v,\T_{t,m,s})$.
        To obtain the desired equality, it suffices to show that $\phi$ is surjective. By inflation--restriction, the map $\phi$ corresponds to the map
        \begin{equation*}
            H^1(K_v^{\ur}/K_v,(\T^{\ac})^{\mathfrak{I}_v})\longrightarrow H^1(K_v^{\ur}/K_v,\T_{t,m,s}^{\mathfrak{I}_v})
        \end{equation*}
        induced by the morphism $(\T^{\ac})^{\mathfrak{I}_v}\to \T_{t,m,s}^{\mathfrak{I}_v}$, where $\mathfrak{I}_v$ is a fixed inertia at $v$. Since $v\nmid Np$, the inertia $\mathfrak{I}_v$ acts trivially on $\T^{\ac}$ and $\T_{t,m,s}$ (see Lemma \ref{lem:residual-representation-Iwasawa-irreducible}). The claim follows by applying the long exact sequence in cohomology to the surjective map $\T^{\ac}\twoheadrightarrow \T_{t,m,s}$ and noticing that $\Gal(K_v^{\ur}/K_v)\cong\hat{\Z}$ has cohomological dimension 1. 

        If $v\mid N$, the proof goes on exactly as in the previous case, where the surjectivity of $(\T^{\ac})^{\mathfrak{I}_v}\to \T_{t,m,s}^{\mathfrak{I}_v}$ is given by the isomorphism $(\T^{\ac})^{\mathfrak{I}_v}\otimes_{\mathcal{R}^{\ac}}\calR_{t,m,s}\cong \T_{t,m,s}^{\mathfrak{I}_v}$ of Proposition \ref{prop:tamagawa-numbers-at-N}.
        
        Let now $v\mid p$.  Let's use the letter $T$ to denote any of $\T^{\ac}$, $\T^{\ac}/(\omega_{2,m})$, $\T^{\ac}/(\omega_{2,m},p^s)$ or $\T^{\ac}/(\omega_{2,m}, p^s, \gamma_{\ac}^{p^t}-1)=\T_{t,m,s}$. Assumption \ref{ass:H.stz} together with Nakayama's lemma (see Lemma \ref{lem:nakayama-for-Galois-representation}) yields $H^0(K_v, F_v^-(T))=\{0\}$. The duality between $F_v^-(T)$ and $F_v^+(T)$ coming from the perfect alternating pairing  of \eqref{eq:pairing} together with Tate local duality implies that
    \begin{equation}
    \label{equ:H-stz-modified}
        H^2(K_v, F_v^+(T))=\{0\},
    \end{equation}
    as noted also in \cite[p. 809]{Buyukboduk:big-Heegner-point-kolyvagin-system}. Since $F_v^+(\T^{\ac})$ is free over $\mathcal{R}^{\ac}$, multiplication by $\omega_{2,s}$ yields the exact sequence 
    \begin{equation*}
        \begin{tikzcd}
	0 & {F_v^+(\T^{\ac})} & {F_v^+(\T^{\ac})} & {F_v^+(\T^{\ac}/(\omega_{2,m}))} & 0.
	\arrow[from=1-1, to=1-2]
	\arrow["{\omega_{2,s}}", from=1-2, to=1-3]
	\arrow[from=1-3, to=1-4]
	\arrow[from=1-4, to=1-5]
        \end{tikzcd}
    \end{equation*}
    Taking the long exact sequence in cohomology, equation (\ref{equ:H-stz-modified}) gives a surjection
    \begin{equation*}
        H^1(K_v,F_v^+(\T^{\ac}))\twoheadrightarrow H^1\big(K_v,F_v^+(\T^{\ac}/(\omega_{2,m}))\big).
    \end{equation*}
    Repeating the same argument to the exact sequences attached to the multiplication by $p^s$ and $\gamma_{\ac}^{p^t}-1$, we eventually obtain a surjection $H^1(K_v,F_v^+(\T^{\ac}))\twoheadrightarrow H^1(K_v,F_v^+(\T_{t,m,s}))$. Then, the commutative diagram with exact rows 
    \begin{equation*}
       \begin{tikzcd}
	{H^1(K_v,F_v^+(\T^{\ac}))} & {H^1(K_v,\T^{\ac})} & {H^1(K_v,F_v^-(\T^{\ac}))} \\
	{H^1(K_v,F_v^+(\T_{t,m,s}))} & {H^1(K_v,\T_{t,m,s})} & {H^1(K_v,F_v^-(\T_{t,m,s}))}
	\arrow[two heads, from=1-1, to=2-1]
	\arrow[from=1-2, to=2-2]
	\arrow[from=1-3, to=2-3]
	\arrow[from=1-1, to=1-2]
	\arrow[from=1-2, to=1-3]
	\arrow[from=2-1, to=2-2]
	\arrow[from=2-2, to=2-3]
\end{tikzcd}
    \end{equation*}
    yields the claim.
\end{proof}

\subsection{Construction of the classes}

In this subsection we pursue a suitable Kolyvagin's descent to the big Heegner classes of  Section \ref{sec:big-heegner-classes}. 

\begin{definition}\label{def:classes-zeta}
        Let $t, n\in\Z_{>0}$ with $n$ prime to $Np$. Define
        \begin{equation*}
		\mathfrak{z}(n)_t:=\Cor_{K_t[n]}^{K[np^{t+1}]} U_p^{-t}\gk_{np^{t+1}}\in H^1(K_t[n],\T^\dag),
		\end{equation*}
        where $U_p\in \mathcal{R}^{\times}$ is the $p$-th Hecke operator. Denote also by $\mathfrak{z}(n)_{t,m,s}$ the image of $\mathfrak{z}(n)_t$ in $H^1(K_t[n],\T_{m,s})$.
\end{definition}

As a direct consequence of the compatibility properties of Proposition \ref{prop:properties-big-Heegner-points}, we obtain the following result.

\begin{lemma}\label{lem:compatibility-of-z}
    Let $t, n\in\Z_{>0}$ with $n$ prime to $Np$ and let $\ell\in\mathcal{P}$ coprime with $n$. Then
    \begin{itemize}
        \item[(a)] $\Cor_{K_{t}[n]}^{K_{t+1}[n]}\big(\mathfrak{z}(n)_{t+1}\big)=\mathfrak{z}(n)_{t}.$
        \item[(b)] $\Cor_{K_t[n]}^{K_t[n\ell]}\big(\mathfrak{z}(n\ell)_{t}\big)=T_\ell \, \mathfrak{z}(n)_{t}$.
    \end{itemize}
\end{lemma}

Fix now $m,s\in\Z_{>0}$, $t\ge 0$ and let $n\in\mathcal{N}_{m,s}$. Recall that $\mathcal{G}_n=\Gal(K[n]/K[1])\cong\Gal(K_t[n]/K_t)\cong \prod_{\ell'\mid n}\mathcal{G}_{\ell}$ and fix a generator $\sigma_{\ell}$ of $\mathcal{G}_{\ell}$. Define the \emph{derivative operators} 
\begin{equation*}
    D_{\ell'}=\sum_{i=1}^{\ell}i\gs_{\ell}^i\in \Z[\mathcal{G}_{\ell}]\quad\text{and}\quad D_n=\prod_{\ell\mid n} D_{\ell}\in\Z[\mathcal{G}_n],
\end{equation*}
where $\ell$ runs over all primes dividing $n$. A straightforward computation (see \cite[(3.5)]{Gross91:Kolyvagins-work}) shows that
\begin{equation}\label{eq:formula-derivative-operators}
    (\gs_{\ell}-1)D_{\ell}=\ell+1-\Tr_{\mathcal{G}_{\ell}}
\end{equation}
in $\Z[\mathcal{G}_{\ell}]$, where $\Tr_{\mathcal{G}_{\ell}}=\sum_{i=0}^{\ell}\gs_{\ell}^i$ is the trace of $\mathcal{G}_{\ell}$. 

\begin{proposition}\label{prop:first-descent-results}
    For every $s'\ge s$ and $\ell\in\mathcal{P}_{m,s'}$ coprime with $n$ we have
    \begin{itemize}
        \item[(a)] $ \Cor_{K_t[n]}^{K_t[n\ell]}(\mathfrak{z}(n\ell)_{t,m,s'})=0$ in $H^1(K_t[n],\T_{m,s'})$.
        \item[(b)] $D_{n\ell}\mathfrak{z}(n\ell)_{t,m,s'}\in H^1(K_t[n\ell],\T_{m,s'})^{\mathcal{G}_\ell}$.
        \item[(c)] The restriction map $H^1(K_t[n], \T_{m,s'})\to H^1(K_t[n\ell],\T_{m,s'})^{\mathcal{G}_\ell}$ is an isomorphism.
    \end{itemize}
\end{proposition}
\begin{proof}
    (a) By Lemma \ref{lem:compatibility-of-z}, one has $\Cor_{K_t[n]}^{K_t[n\ell]}(\mathfrak{z}(n\ell)_{t,m,s'})=T_{\ell} \,\mathfrak{z}(n)_{t,m,s'}$ and Lemma \ref{lem:Tl-acts-as-0} implies that $T_{\ell}$ is the zero operator on $H^1(K_t[n],\T_{m,s'})$.

    (b) It suffices to prove that $(\gs_{\ell}-1)D_{n\ell}\mathfrak{z}(n\ell)_{t,m,s'}=0$
    in $H^1(K_t[n],\T_{m,s'})$. Using equation \eqref{eq:formula-derivative-operators} we obtain that
		\begin{equation*}
		(\gs_{\ell}-1)D_{n\ell}\mathfrak{z}(n\ell)_{t,m,s'}=(\ell+1-\Tr_{\mathcal{G}_{\ell}})D_{n}\mathfrak{z}(n\ell)_{t,m,s'}.
		\end{equation*}
		Since $\ell\in\mathcal{P}_{m,s'}$, we have that $\ell+1$ is zero in $\calR_{m,s'}$. By  \cite[Corollary 1.5.7]{Neukirch-Schmidt-Wingberg:cohomology-of-number-fields}, one has that $\Tr_{\mathcal{G}_{\ell}}=\res_{K_t[n]}^{K_t[n\ell]}\circ \Cor_{K_t[n]}^{K_t[n\ell]}$ and we conclude by applying point (a).

    (c) By Shapiro's lemma, we have that $H^0(K_t[n\ell],\T_{m,s'})\cong H^0(K[n\ell],\T_{t,m,s'})$, which is trivial thanks to Lemma \ref{lem:no-invariants}. We conclude by applying the inflation--restriction exact sequence.
\end{proof}

As a consequence of this proposition, we have that $D_n\mathfrak{z}(n)_{t,m,s}\in H^1(K_t[n],\T_{m,s})^{\mathcal{G}_n}$ and that the restriction map $H^1(K_t[1], \T_{m,s})\to H^1(K_t[n],\T_{m,s})^{\mathcal{G}_n}$ is an isomorphism.

\begin{definition}\label{def:derived-classes}
Define $\boldsymbol{\gk}'(n)_{t,m,s}:=\Cor_{K_t}^{K_t[1]}\big(\res_{K_t[1]}^{K_t[n]}\big)^{-1}D_n\mathfrak{z}(n)_{t,m,s}\in H^1(K_t, \T_{m,s})$.
\end{definition}

Recalling the isomorphism $\sh_t\colon H^1(K_t,\T_{m,s})\cong H^1(K,\T_{t,m,s})$ of \eqref{eq:shapiro-isomorphisms}, we are able to give the following final definition. 

\begin{definition}
    For every $n\in\mathcal{N}_{m,s}$ and $t\ge 0$, define
    \begin{equation*}
            \boldsymbol{\gk}(n)_{t,m,s}:=\sh_t(\boldsymbol{\gk}'(n)_{t,m,s})\in H^1(K,\T_{t,m,s}).
    \end{equation*}
\end{definition}

For every $\ga\ge t$ let $\psi_{\ga,t}\colon H^1(K,\T_{\ga,m,s})\to H^1(K,\T_{t,m,s})$ be the natural projection. We recover the following compatibility in the anticyclotomic tower.

    \begin{lemma}\label{lem:vertical-compatibility-kn}
	    For every $\ga> t\ge 0$ we have that $\psi_{\ga,t}(\boldsymbol{\gk}(n)_{\ga,m,s})=\boldsymbol{\gk}(n)_{t,m,s}$.
	\end{lemma}
	\begin{proof}
		By Lemma \ref{lem:compatibility-of-z} we have that $\Cor_{K_{t}[n]}^{K_{t+1}[n]}(\mathfrak{z}(n)_{\ga})=\mathfrak{z}(n)_{t}$. Corestriction commutes with the action of $D_n$ (see \cite[Proposition 1.5.4]{Neukirch-Schmidt-Wingberg:cohomology-of-number-fields}), with reduction to $\T_{m,s}$ and, since $K_\ga[1]$ and $K_t[n]$ are disjoint over $K_t[1]$, also with the inverse of the restriction defining $\boldsymbol{\gk}'(n)_{\ga,m,s}$ and  $\boldsymbol{\gk}'(n)_{t,m,s}$ (see \cite[Corollary 1.5.8]{Neukirch-Schmidt-Wingberg:cohomology-of-number-fields}). Therefore, we obtain that
        \begin{equation*}
            \Cor_{K_\ga/K_t}(\boldsymbol{\gk}'(n)_{\ga,m,s})=\boldsymbol{\gk}'(n)_{t,m,s}.
        \end{equation*}
        The commutative diagram \eqref{eq:diagram-Shapiro-lemma} with $F=K$, $F_\ga=K_\ga$, $F_t=K_t$ and $T=\T_{m,s}$ yields the claimed equality.
\end{proof}

\subsection{Kolyvagin primes and main result}\label{sec:kolyvagin-primes-and-main-result}

The aim of the rest of this chapter is to show that a slight modification of the classes $\boldsymbol{\gk}(n)_{t,m,s}$ form a modified universal Kolyvagin system. In order to do this, besides our running assumptions, we will need the following big image assumption on the representation $\T^{\dag}$. 

\begin{assumption}\label{ass:big-image}
    The image of $G_\Q$ in $\Aut_{\mathcal{R}}(\T^{\dag})$ contains the scalars $\Z_p^\times$.
\end{assumption}

\begin{remark}
    Let us make a few comments on this assumption.
    \begin{itemize}
        \item[(i)] In what follows, Assumption \ref{ass:big-image} is used in the proof of Lemmas \ref{lem:finer-choice-of-primes} and \ref{lem:Howard-result}. In order to obtain Theorems \ref{thm:A} and \ref{thm:B}, one could relax this and assume that the image of $G_\Q$ in $\Aut_{\mathcal{R}}(\T^{\dag})$ contains the scalars $(\Z_p^\times)^n$, for some natural number $n$ not divisible by $p-1$. However, this would make the notation and the arguments of the next pages heavier, so we leave this slight generalization to the interested reader.
        \item[(ii)] There is an infinite class of Hida families for which Assumption \ref{ass:big-image} holds. For example, this holds for infinitely many primes $p$ if $N$ is square-free and the fixed modular form $f$ has weight $2$ and trivial character. Indeed, by \cite[Theorem 4.15]{Vigni22:shafarevich-tate-groups}, the image of $G_\Q$ in $\Aut_{\mathcal{R}}(\T^{\dag})$ contains $\SL_2(\mathcal{R})$ whenever $p$ lies in an explicit set of primes of density $1$. By Proposition \ref{prop:characteristic-polynomial-Frobenius}, if $\ell\nmid Np$ then $\Fr_\ell$ has characteristic polynomial with determinant $\ell$. Combining this two facts with the density of the set of primes $\ell\nmid Np$ in $\Z_p$, we conclude that the image of $G_\Q$ contains all elements of $\Aut_{\mathcal{R}}(\T^{\dag})$ whose determinant lies in $\Z_p$.
        \item[(iii)] If one assumes that the residual representation attached to $f$ contains $\SL_2(\F_q)$ for some $q\ge 7$, it is possible to compute explicit subgroups of $\Aut_{\mathcal{R}}(\T^{\dag})$ that lie in the image of $G_\Q$, by combining \cite[Corollary 11.2]{Conti-Lang-Medvedovsky23:Big-images} with an argument of Papier (see \cite[Section 4]{Ribet_1985}).
    \end{itemize}
\end{remark}

We now need to shrink the set $\mathcal{P}_{m,s}$ of admissible primes. As a consequence of Lemma \ref{lem:Tl-acts-as-0}, the two quantities $\ell+1\pm T_\ell$ are divisible by $p^s$ in $\calR_m$. This motivates the following definition.

\begin{definition}\label{dfn:kolyvagin-primes}
    For every $m,s\in\Z_{>0}$ let $\mathcal{L}_{m,s}$ be the set of primes of $\mathcal{P}_{m,s}$ such that the two quantities $\frac{\ell+1\pm T_\ell}{p^s}$ are units in $\calR_m$. Define also
    \begin{equation*}
        \mathcal{P}':=\bigcup_{s\ge m>0} \mathcal{L}_{m,s}.
    \end{equation*}
    Set $\mathcal{P}_{m,s}':=\mathcal{P}'\cap\mathcal{P}_{m,s}$ and let $\mathcal{N}_{m,s}'$ be the set of all square-free products of elements of $\mathcal{P}_{m,s}'$.
\end{definition}

As a consequence of Assumption \ref{ass:big-image}, we obtain the following result.

\begin{lemma}\label{lem:finer-choice-of-primes}
    For every $m,s\in\Z_{>0}$ with $s\ge m$ the set $\mathcal{L}_{m,s}$ is infinite.
\end{lemma}
\begin{proof}
    First, notice that an element $r\in\mathcal{R}_m$ is a unit if and only if its projection to one (hence any) nonzero quotient of $\calR_m$ is a unit.

    Fix a prime $\ell_0\in\mathcal{P}_{m,s}$. By Lemma \ref{lem:charachteristic-polynomial-frobenius-Tm}, we have that
    \begin{equation*}
        \Tr(\Fr_{\ell_0}\mid \T_m)=(-1)^{\frac{k+j}{2}-1}T_{\ell_0}\quad\text{and}\quad\det(\Fr_{\ell_0}\mid \T_m)=\ell_0.
    \end{equation*}
    Let now $\ga\in 1+p^s\Z_p$ and $\gs_\ga\in G_\Q$ such that its image in $\Aut(\T_m)$ is the scalar $\ga$, whose existence is granted by Assumption \ref{ass:big-image}. Then one obtains
    \begin{equation*}
        \Tr(\Fr_{\ell_0}\gs_\ga\mid \T_m)=(-1)^{\frac{k+j}{2}-1}\ga T_{\ell_0}\quad\text{and}\quad\det(\Fr_{\ell_0}\gs_\ga\mid \T_m)=\ga^2\ell_0.
    \end{equation*}
    Let $\mathcal{L}_{m,s}(\ell_0, \alpha)$ be the set consisting of those primes $\ell$ such that $\Fr_\ell$ is conjugated to $\Fr_{\ell_0}\gs_\ga$ in $\Gal(K(\T_{m,2s})/\Q)$. By Chebotarev's density theorem, the set $\mathcal{L}_{m,s}(\ell_0, \alpha)$ has infinite cardinality. For any $\ell\in\mathcal{L}_{m,s}(\ell_0, \alpha)$, one has
    \begin{equation}\label{eq:proof-finer-set-of-primes}
        T_{\ell}\equiv \ga T_{\ell_0}\quad\text{and}\quad \ell\equiv\ga^2\ell_0
    \end{equation}
    as elements of $\calR_{m,2s}$. Reducing to $\calR_{m,s}$ the above congruences and using the fact that $\ga\equiv 1\bmod p^s$, one easily sees that $\mathcal{L}_{m,s}(\ell_0, \alpha)\subseteq\mathcal{P}_{m,s}$. 

    We now look for an $\ga\in 1+p^s\Z_p$ such that the division of the two elements
    \begin{equation}\label{eq:elements-to-be-divided}
        \ga^2\ell_0+1\pm\ga T_{\ell_0}
    \end{equation}
    by $p^s$ is a unit of $\mathcal{R}_{m,2s}$. The existence of such an $\ga$ would imply that the set of primes $\mathcal{L}_{m,s}(\ell_0,\ga)\subseteq\mathcal{L}_{m,s}$, since, by \eqref{eq:proof-finer-set-of-primes}, we have the congruences
    \begin{equation*}
        \ell+1\pm T_\ell\equiv \ga^2\ell_0+1\pm\ga T_{\ell_0}
    \end{equation*}
    in $\calR_{m,2s}$, and this would conclude the proof.

    Writing $\alpha=1+p^{s}x$ for $x\in\Z_p$, dividing the elements \eqref{eq:elements-to-be-divided} by $p^s$ and reducing them modulo the maximal ideal of $\mathcal{R}_m$, one finds that in order to get such an $\alpha$ it is sufficient to require that the reduction of $x$ avoids two explicit values. Since $p>2$, this is always possible.
\end{proof}

As a consequence of this lemma, we also have that the sets $\mathcal{P}_{m,s}'$ are infinite. If $\ell\in\mathcal{P}_{m,s}'$, let $s(\ell)$ be the unique positive integer such that such that $\ell\in \mathcal{L}_{m,s(\ell)}$. In this case, thanks to Lemma \ref{lem:finer-choice-of-primes}, the values $\frac{\ell+1\pm T_\ell}{p^{s(\ell)}}$ are units in $\mathcal{R}_m$. Moreover, we have the following result.

\begin{corollary}\label{cor:finer-choice-of-primes}
    If $\ell\in\mathcal{P}_{m,s}'$, the image of $\frac{(\ell+1)\Fr_\ell\pm T_\ell}{p^{s(\ell)}}$ in $\Aut(\T_{m,s})$ is invertible.
\end{corollary}
\begin{proof}
    By point (b) of Definition \ref{dfn:admissible-primes}, we have that $\Fr_\ell$ is conjugated to the complex conjugation $\tau_c$ in $\Gal(K(\T_{m,s})/\Q)$. Therefore, we can find an $\calR_{m,s}$-basis of $\T_{m,s}$ made of eigenvectors (with eigenvalues $1$ and $-1$) for the action of $\Fr_\ell$. Using this basis one can easily compute the determinants of the morphisms $\frac{(\ell+1)\Fr_\ell\pm T_\ell}{p^{s(\ell)}}$ and conclude that they are units.
\end{proof}

The aim of the rest of this chapter will be to prove the following result.

\begin{theorem}\label{thm:main-thm}
    Under Assumptions \ref{ass:residual-representation-irreducible}, \ref{ass:field-K}, \ref{ass:tamagawa-factors-at-N}, \ref{ass:H.stz} and \ref{ass:big-image}, there is a set of automorphisms $\{\chi_{n,\ell}\}$ of $\T^{\ac}$ and a universal Kolyvagin system $\boldsymbol{\gk}^{\ac}\in\overline{\KS}(\T^{\ac},\mathcal{F}_{\Gr},\mathcal{P}',\{\chi_{n,\ell}\})$ such that
    \begin{equation*}
        \boldsymbol{\gk}^{\ac}(1)=\varprojlim_t\big(\Cor^{K_t[1]}_{K_t}\mathfrak{z}(1)_{t}\big)\in H^1(K,\T^{\ac}).
    \end{equation*}
\end{theorem}

\begin{remark}
    The classes $\Cor^{K_t[1]}_{K_t}\mathfrak{z}(1)_{t}$ form a projective system thanks to point (a) of Lemma \ref{lem:compatibility-of-z}. The strategy for the proof of Theorem \ref{thm:main-thm} will be the following:
    \begin{itemize}
        \item[(i)] Show that $\boldsymbol{\kappa}(n)_{t,m,s}\in H^1_{\mathcal{F}_{\Gr}(n)}(K,\T_{t,m,s})$ for all $n\in\mathcal{N}_{m,s}$.
        \item[(ii)] Find an explicit relation between $\loc_\gl \boldsymbol{\kappa}(n\ell)_{t,m,s}$ and $\loc_\gl  \boldsymbol{\kappa}(n)_{t,m,s}$ for every $n\ell\in\mathcal{N}_{m,s}'$, where $\gl$ is the prime of $K$ above $\ell$ and $\loc_\gl$ is the localization map at $\gl$ in cohomology.
        \item[(iii)] Define $\boldsymbol{\gk}^{\ac}(n)_{t,m,s}:=\boldsymbol{\gk}(n)_{t,m,s}\otimes\bigotimes_{\ell'\mid n}\gs_{\ell'}\in H^1_{\mathcal{F}_{\Gr}}(K,\T_{t,m,s})\otimes\mathcal{G}(n)$ and show that these classes interpolate into a modified universal Kolyvagin system with the required properties.
    \end{itemize}
\end{remark}

\subsection{Local properties of the classes}

Before studying the local properties of the classes $\boldsymbol{\kappa}(n)_{t,m,s}$, we set some notation. If $L/F$ is a Galois extension of number fields, $v$ is a prime of $F$ and $T$ is a $G_F$-module, we write
\begin{equation*}
    H^1(L_v,T):=\bigoplus_{w\mid v} H^1(L_w,T)\quad\text{and}\quad H^1(L_v^{\ur}, T):=\bigoplus_{w\mid v} H^1(L_w^{\ur},T),
\end{equation*}
where $f$ is the residue degree of the extension $L/K$ at $v$. We will also denote by
\begin{equation*}
    \loc_v:=\bigoplus_{w\mid v}\loc_w\colon H^1(L,T)\to H^1(L_v,T)
\end{equation*}
for the corresponding semi-localization map. We will often use the fact that semi-localization commutes with restriction, corestriction, Galois action and Shapiro's map. For the first three, see \cite[§1.5]{Neukirch-Schmidt-Wingberg:cohomology-of-number-fields}. For the behaviour with respect to Shapiro's map, see \cite[§3.1.2]{Skinner-Urban:Iwasawa.main-conjecture-for-GL2}.

\begin{proposition}\label{prop:z-in-the-Selmer}
	If  $m,s\in\Z_{>0}$, $t\ge 0$ and $n\in\mathcal{N}_{m,s}$ then $\mathfrak{z}(n)_{t}\in H^1_{\mathcal{F}_{\Gr}}(K_t[n],\T^{\dag})$.
\end{proposition}
\begin{proof}
    Following the proof of \cite[Proposition 10.1]{Longo-Vigni11:quaternion-algebras-Hida-families} (that in turn relies on \cite[Proposition 2.4.5]{Howard07:variation-of-Heegner-points-in-Hida-families}), one shows that $\loc_w\kappa_{np^{t}}\in H^1_{\mathcal{F}_{\Gr}}(K[np^{t}]_w,\T^\dag)$ for every prime $w\nmid N^-$ of $K[np^{t}]$. When $w\mid N^-$, one is able to show that the restriction of $\loc_w\kappa_{np^{t}}$ to $H^1(K[np^{t}]_w^{\ur},\T^\dag)$ is $\mathcal{R}$-torsion. By class field theory, $K_{v}=K[np^{t}]_w$ where $v=w\cap K$. By Assumption \ref{ass:tamagawa-factors-at-N}, we know that $H^1(K_{v}^{\ur},\T^\dag)$ is $\mathcal{R}$-torsion-free, hence the restriction of $\loc_w\kappa_{np^{t}}$ to it is zero. Therefore,
    \begin{equation}\label{eq:LV-classes-in-the-Selmer}
        \kappa_{np^t}\in H^1_{\mathcal{F}_{\Gr}}(K[np^{t}],\T^{\dag}).
    \end{equation}
		Fix now a prime $v\nmid p$ of $K_t[n]$. A careful study of restriction and corestriction maps based on \cite[Proposition 1.5.6]{Neukirch-Schmidt-Wingberg:cohomology-of-number-fields} (see also \cite[§B.2]{mastella-zerman2025:anticyclotomic}) yields the commutative diagram
		\begin{equation*}
        \begin{tikzcd}
	{H^1(K[np^{t+1}],\T^{\dag})} & {H^1(K[np^{t+1}]_v,\T^{\dag})} & {H^1(K[np^{t+1}]_v^{\ur},\T^{\dag})} \\
	{H^1(K_t[n],\T^{\dag})} & {H^1(K_t[n]_v,\T^{\dag})} & {H^1(K_t[n]_v^{\ur},\T^{\dag}).}
	\arrow["{U_p^{-t}\circ\Cor}"', from=1-1, to=2-1]
	\arrow["{\loc_v}", from=1-1, to=1-2]
	\arrow["{\loc_v}", from=2-1, to=2-2]
	\arrow[from=1-2, to=2-2]
	\arrow["\Res", from=1-2, to=1-3]
	\arrow[from=1-3, to=2-3]
	\arrow["\Res", from=2-2, to=2-3]
        \end{tikzcd}
            \end{equation*}
		Using \eqref{eq:LV-classes-in-the-Selmer} and chasing the above diagram, we conclude that $\loc_v(\mathfrak{z}(n)_{t})\in H^1_{\mathcal{F}_{\Gr}}(K_t[n]_v,\T^{\dag})$.
		
    Similarly, if $v\mid p$, the commutative diagram
		\begin{equation*}
		\begin{tikzcd}
	{H^1(K[np^{t+1}],\T^\dag)} & {H^1(K[np^{t+1}]_v,\T^{\dag})} & {H^1(K[np^{t+1}]_v,F^-_p(\T^{\dag}))} \\
	{H^1(K_t[n],\T^{\dag})} & {H^1(K_t[n]_v,\T^{\dag})} & {H^1(K_t[n]_v,F^-_p(\T^{\dag}))}
	\arrow["{\loc_v}", from=1-1, to=1-2]
	\arrow["{U_p^{-t}\circ\Cor}"', from=1-1, to=2-1]
	\arrow[from=1-2, to=1-3]
	\arrow[from=1-2, to=2-2]
	\arrow[from=1-3, to=2-3]
	\arrow["{\loc_v}", from=2-1, to=2-2]
	\arrow[from=2-2, to=2-3]
            \end{tikzcd}
		\end{equation*}
		together with \eqref{eq:LV-classes-in-the-Selmer} yields that $\loc_v(\mathfrak{z}(n)_{t})\in H^1_{\mathcal{F}_{\Gr}}(K_t[n]_v,\T^{\dag})$.
\end{proof}

\begin{theorem}\label{thm:kolyvagin-classes-lie-in-the-Selmer}
    Let $m,s\in\Z_{>0}$, $t\ge0$ and $n\in\mathcal{N}_{m,s}$. Then $\boldsymbol{\kappa}(n)_{t,m,s}\in H^1_{\mathcal{F}_{\Gr}(n)}(K,\T_{t,m,s})$.
\end{theorem}
\begin{proof}
    We have to show that $\loc_v\boldsymbol{\kappa}(n)_{t,m,s}\in H^1_{\mathcal{F}_{\Gr}(n)}(K_v,\T_{t,m,s})$ for every prime $v$ of $K$. Let $v_n$ be a prime of $K_t[n]$ above $v$ and call $v_1=v_n\cap K_t[1]$.

    (i) Suppose that $v\nmid np$, so that $H^1_{\mathcal{F}_{\Gr}(n)}(K_v,\T_{t,m,s})=H^1_{\ur}(K_v,\T_{t,m,s})$ by Lemma \ref{lem:easier-description-Greenberg-condition}. Starting from Proposition \ref{prop:z-in-the-Selmer}, the functoriality of $D_n$, restriction, corestriction (see \cite[§1.5]{Neukirch-Schmidt-Wingberg:cohomology-of-number-fields}) and Shapiro's lemma (see \cite[§3.1.2]{Skinner-Urban:Iwasawa.main-conjecture-for-GL2}) in (semi-)local Galois cohomology together with the fact that $K_v^{\ur}=K_t[1]_{v_1}^{\ur}=K_t[n]_{v_n}^{\ur}$ easily imply that $\loc_v\boldsymbol{\kappa}(n)_{t,m,s}\in H^1_{\ur}(K_v,\T_{t,m,s})$.

    (ii) Suppose that $v\mid n$, so that $H^1_{\mathcal{F}_{\Gr}(n)}(K_v,\T_{t,m,s})=H^1_{\tr}(K_v,\T_{t,m,s})$. Let $\ell$ be the rational prime below $v$ and set $v_\ell=v_n\cap K_t[\ell]$. By Proposition \ref{prop:z-in-the-Selmer} and the functoriality of $D_n$, we obtain that $\loc_{v_n}(D_n\mathfrak{z}(n)_{t,m,s})\in H^1_{\ur}(K_t[n]_{v_n},\T_{m,s})$. One can lift the fixed generator $\sigma_\ell$ of $\mathcal{G}_\ell$ to an element of $G_{K_t[1]_{v_1}^{\ur}}$, that acts trivially on $\T_{m,s}$ since $\T_{m,s}$ is unramified outside $Np$. Then, using the fact that $\loc_{v_n}\mathfrak{z}(n)_{t,m,s}$ is inflated by an unramified cocycle (see Lemma \ref{lem:app-unramified-cocycles}), we have that
    \begin{equation*}
       (D_\ell \mathfrak{z}(n)_{t,m,s})(\Fr_{v_n})=\sum_{i=1}^\ell i(\sigma_\ell \mathfrak{z}(n)_{t,m,s})(\Fr_{v_n})=\sum_{i=1}^\ell i \mathfrak{z}(n)_{t,m,s}(\Fr_{v_n})=\frac{\ell(\ell+1)}{2} \mathfrak{z}(n)_{t,m,s}(\Fr_{v_n})
   \end{equation*}
     is trivial, since $\ell+1$ is zero in $\calR_{m,s}$. By \cite[Lemma 1.2.1]{Mazur-Rubin:Kolyvagin-systems}, the evaluation at $\Fr_{v_n}$ induces an isomorphism between $H^1_{\ur}(K_t[n]_{v_n},\T_{m,s})$ and $\T_{m,s}$, therefore we have shown that
    \begin{equation}\label{eq:step-1}
        \loc_{v_n}(D_n\mathfrak{z}(n)_{t,m,s})=0.
    \end{equation}
    By class field theory, the prime $v_\ell$ splits completely in $K_t[n]$, hence $K_t[n]_{v_n}=K_t[\ell]_{v_\ell}$. Then, the commutative diagram
\begin{equation*}
		\begin{tikzcd}
		{H^1(K_t[n],\T_{m,s})} && {H^1(K_t[n]_{v_n},\T_{m,s})} \\
		{H^1(K_t[1],\T_{m,s})} & {H^1(K_t[1]_{v_1},\T_{m,s})} & {H^1(K_t[\ell]_{v_\ell},\T_{m,s})}
		\arrow["{\loc_{v_n}}", from=1-1, to=1-3]
		\arrow["\res", from=2-1, to=1-1]
		\arrow["{\loc_{v_1}}", from=2-1, to=2-2]
		\arrow["\res", from=2-2, to=2-3]
		\arrow[Rightarrow, no head, from=2-3, to=1-3]
		\end{tikzcd}
		\end{equation*}
together with \eqref{eq:step-1} and the functoriality of corestriction and Shapiro's map in semi-local Galois cohomology yield that the restriction of $\loc_v\boldsymbol{\kappa}(n)_{t,m,s}$ to $H^1(K[\ell]_{v_\ell},\T_{t,m,s})$ is zero. Thanks to the discussion of Remark \ref{rk:transverse-condition}, we conclude that $\loc_v\boldsymbol{\kappa}(n)_{t,m,s}\in H^1_{\tr}(K_v,\T_{t,m,s})$.

(iii) Suppose that $v\mid p$, so that 
\begin{equation*}
    H^1_{\mathcal{F}_{\Gr}(n)}(K_v,\T_{t,m,s})=\ker\big(H^1(K_v,\T_{t,m,s})\to H^1(K_v,F_v^-(\T_{t,m,s}))\big)
\end{equation*}
by Lemma \ref{lem:easier-description-Greenberg-condition}. For every $\ga\ge t$, choose a place $v_n(\ga)$ above $v_n$ with the property that $v_n(\ga+1)\cap K_\ga[n]=v_n(\ga)$, and set $v_1(\ga)=v_n(\ga)\cap K_\ga[1]$. The functoriality of $D_n$ together with Proposition \ref{prop:z-in-the-Selmer} implies that the image of $D_n\mathfrak{z}(n)_{\ga}$ in $H^1(K_\ga[n]_{v_n(\ga)},F_v^-(\T_{m,s}))$ is zero. Define $\tilde{\boldsymbol{\gk}}(n)_{\ga,m,s}:=(\res^{K_\ga[n]}_{K_\ga[1]})^{-1}D_n\mathfrak{z}(n)_{\ga,m,s}\in H^1(K_\ga[1],\T_{m,s})$ and call $c(n)_{\ga,m,s}$ the image of $\loc_{v_1(\ga)}\tilde{\boldsymbol{\gk}}(n)_{\ga,m,s}$ in $H^1(K_\ga[1]_{v_1(\ga)},F_v^-(\T_{m,s}))$. Then, we know that $c(n)_{\ga,m,s}$ lies in
\begin{equation*}
    M_\ga:=\ker\Big( H^1(K_\ga[1]_{v_1(\ga)},F_v^-(\T_{m,s}))\overset{\res}{\longrightarrow} H^1(K_\ga[n]_{v_n(\ga)},F_v^-(\T_{m,s}))\Big).
\end{equation*}

As a consequence of the norm compatibilities of Lemma \ref{lem:compatibility-of-z}, one can show that
\begin{equation*}
    \Cor^{K_{\ga+1}[1]_{v_1(\ga+1)}}_{K_\ga[1]_{v_1(\ga)}} c(n)_{\ga+1,m,s}=c(n)_{\ga,m,s}
\end{equation*}
for every $\ga\ge t$. By applying the inflation--restriction exact sequence, we obtain that $M_\ga$ is isomorphic to
\begin{equation*}
    N_\ga:=H^1\big(K_\ga[n]_{v_n(\ga)}/K_\ga[1]_{v_1(\ga)},H^0(K_\ga[n]_{v_n(\ga)},F_v^-(\T_{m,s})\big)
\end{equation*}
and that the corestriction map $M_{\ga+1}\to M_\ga$ corresponds to the map
\begin{equation}\label{eq:explicit-corestriction-inflation}
    [\xi]\mapsto [\Tr_{K_{\ga+1}[n]_{v_n(\ga+1)}/K_{\ga}[n]_{v_n(\ga)}}\circ \xi\circ i_{\ga}^{-1}]
\end{equation}
on cocycle classes representing elements of $N_{\ga+1}$, where $i_\ga$ is the natural isomorphism 
\begin{equation*}
    i_\ga\colon \Gal(K_{\ga+1}[n]_{v_n(\ga+1)}/K_{\ga+1}[1]_{v_1(\ga+1)})\to\Gal(K_{\ga}[n]_{v_n(\ga)}/K_{\ga}[1]_{v_1(\ga)}).
\end{equation*}
Since the module $\T_{m,s}$ is finite, the cardinality of the modules $H^0(K_\ga[n]_{v_n(\ga)},F_v^-(\T_{m,s}))$ is bounded independently on $\ga$, hence there is $\ga_0\ge t$ such that these modules stabilize for all $\ga\ge\ga_0$. For such $\ga$, the trace map appearing in \eqref{eq:explicit-corestriction-inflation} coincides with multiplication by $p$. Therefore, there is a big enough $\tilde{\ga}$ such that
\begin{equation*}
    \Cor^{K_{\tilde{\ga}}[1]_{v_1(\tilde{\ga})}}_{K_t[1]_{v_1}}\colon M_{\tilde{\ga}}\to M_t
\end{equation*}
is the zero map. Since $c(n)_{t,m,s}$ lies in the image of this map, it is trivial. This means that $\loc_{v_1}\tilde{\boldsymbol{\gk}}(n)_{\ga,m,s}$ lies in
\begin{equation*}
    \ker\big(H^1(K_t[1]_{v_1},\T_{m,s})\to H^1(K_t[1]_{v_1},F_v^{-}(\T_{m,s})\big).
\end{equation*}
One concludes by applying the functoriality of corestriction and Shapiro's map (see \cite[§3.1.2]{Skinner-Urban:Iwasawa.main-conjecture-for-GL2}) in (semi-)local Galois cohomology.
\end{proof}

\subsection{The key relation}\label{sec:the-key-relation}

Recall now the morphisms $\ga_\gl$ and $\beta_\gl$ defined in \eqref{eq:definition-finite-singular}. The main step needed to end the proof of Theorem \ref{thm:main-thm} is the following relation, whose proof is inspired by the arguments of \cite[§7-9]{Nekovar92:chow-groups} and \cite[§1.7]{Howard04:heegner-point-kolyvagin-system}. For every class $\boldsymbol{x}\in H^1(K_\gl,\T_{t,m,s})$, we indicate by $[\boldsymbol{x}]_{\s}$ its image in the singular quotient $H^1_\s(K_\gl,\T_{t,m,s})$.

\begin{lemma}\label{lem:key-formula}
    Let $m,s\in\Z_{>0}$ and $t\ge 0$. For any $n\ell\in\mathcal{N}_{m,s}'$ with $\ell$ prime there is $\vartheta_\ell\in\Aut(\T_{t,m,s})$ such that
    \begin{equation*}
        \ga_\gl\big(\loc_\gl \boldsymbol{\gk}(n)_{t,m,s}\big)=\vartheta_\ell\Big(\beta_\gl\big([\loc_\gl \boldsymbol{\gk}(n\ell)_{t,m,s}]_{\s}\otimes\sigma_\ell\big)\Big)
    \end{equation*}
    as elements of $\T_{t,m,s}$, where $\gl$ is the prime of $K$ above $\ell$.
\end{lemma}
\begin{proof}
    Recall that, since $\ell\in\mathcal{P}_{m,s}'$, there is an index $s(\ell)\ge s$ such that $s(\ell)\ge m$ and $\ell\in \mathcal{L}_{m,s(\ell)}$. To ease the notation, let $\boldsymbol{x}$ and $\bar{\boldsymbol{x}}$ be the image of $D_n \mathfrak{z}(n)_{t}$ in $H^1(K_t[n],\T_m)$ and $H^1(K_t[n],\T_{m,s(\ell)})$, respectively, and let $\boldsymbol{y}$ and $\bar{\boldsymbol{y}}$ be the image of $D_n \mathfrak{z}(n\ell)_{t}$ in $H^1(K_t[n\ell],\T_m)$ and $H^1(K_t[n\ell],\T_{m,s(\ell)})$, respectively. We will also use the letters $x$, $\bar{x}$, $y$, $\bar{y}$ to denote cocycles representing $\boldsymbol{x}$, $\bar{\boldsymbol{x}}$, $\boldsymbol{y}$ and $\bar{\boldsymbol{y}}$, respectively.

    By point (a) of Proposition \ref{prop:first-descent-results}, we know that $\Cor_{K_t[n]}^{K_t[n\ell]}(\bar{\boldsymbol{y}})=0$. Therefore, there is an element $\bar{a}\in \T_{m,s(\ell)}$ such that
    \begin{equation*}
        \big(\Cor_{K_t[n]}^{K_t[n\ell]}(\bar{y})\big)(g)=(g-1)\bar{a}
    \end{equation*}
    for every $g\in\Gal(\bar{\Q}/K_t[n])$, where the corestriction map acts on cocycles via the explicit description of \cite[p. 48]{Neukirch-Schmidt-Wingberg:cohomology-of-number-fields}. The element $\bar{a}$ is uniquely determined by $\bar{y}$, since $H^0(K_t[n],\T_{m,s(\ell)})=\{0\}$ by the proof of point (c) of Proposition \ref{prop:first-descent-results}. The cohomology class $D_\ell \bar{\boldsymbol{y}}$ can be represented by the cocycle $D_\ell \bar{y}\in Z^1(K_t[n\ell],\T_{m,s(\ell)})$ defined as
    \begin{equation*}
        (D_\ell \bar{y})(h):=\sum_{i=1}^\ell i\tilde{\gs}_\ell^i\bar{y}(\tilde{\gs}_\ell^{-i}h\tilde{\gs}_\ell^i)
    \end{equation*}
    for every $h\in\Gal(\bar{\Q}/K_t[n\ell])$, where $\tilde{\gs}_\ell$ is a lift of the fixed generator $\sigma_\ell$ of $\mathcal{G}_\ell$ to $\Gal(\bar{\Q}/K_t[n])$. By points (b) and (c) of Proposition \ref{prop:first-descent-results}, we know that $D_\ell \bar{\boldsymbol{y}}$ lies in the image of the restriction map with domain $H^1(K_t[n],\T_{m,s(\ell)})$. At the level of cocycles, in \cite[§7]{Nekovar92:chow-groups} Nekov\' a\v r showed that there is a cocycle $f_{\bar{y}}\in Z^1(K_t[n],\T_{m,s(\ell)})$ uniquely determined by the conditions
    \begin{equation}\label{eq:definition-fy}
        \res_{K_t[n]}^{K_t[n\ell]} f_{\bar{y}}=D_\ell\bar{y}\quad\text{and}\quad f_{\bar{y}}(\tilde{\gs}_\ell)=-\tilde{\gs}_\ell\bar{a}.
    \end{equation}
    Moreover, its cohomology class $[f_{\bar{y}}]\in H^1(K_t[n],\T_{m,s(\ell)})$ is independent on the choice of the representative $\bar{y}$ of $\bar{\boldsymbol{y}}$.

    Let now $\gl_{n\ell}$ be a prime of $K_t[n\ell]$ above $\gl$, and set $\gl_n:=\gl_{n\ell}\cap K_t[n]$. Notice that, by class field theory, $K_t[n]_{\gl_n}=K_\gl$ and $K_t[n\ell]_{\gl_{n\ell}}=K[\ell]_{\gl_\ell'}$, where $\gl_\ell':=\gl_{n\ell}\cap K[\ell]$. The action of $\Gal(\bar{\Q}_\ell/K_\gl)$ on $\T_{m,s(\ell)}$ is trivial as $\T_{m,s(\ell)}$ is unramified at $\ell$ and $\Fr_\gl$ acts trivially on $\T_{m,s(\ell)}$, therefore $H^1(K_\gl,\T_{m,s(\ell)})$ and $H^1(K[\ell]_{\gl_{\ell}'}/K_\gl,\T_{m,s(\ell)})$ are both homomorphism groups.
    
    The same argument used to prove \eqref{eq:step-1} implies that $\loc_{\gl_{n\ell}}(D_\ell\bar{\boldsymbol{y}})=0$. By applying the inflation--restriction exact sequence, we find a cohomology class $\bar{\boldsymbol{z}}_0\in H^1(K[\ell]_{\gl_{\ell}'}/K_\gl,\T_{m,s(\ell)})=\Hom(\mathcal{G}_\ell,\T_{m,s(\ell)})$ such that $\inf_{K_\gl}^{K[\ell]_{\gl_{\ell}'}}\bar{\boldsymbol{z}}_0=\loc_{\gl_n}f_{\bar{y}}$. Choose now a lift $\tilde{\gs}_\ell$ of $\sigma_\ell$ that lies in $\Gal(\bar{\Q}_\ell/\bar{\Q}_\ell^{\ur})$ (using the fixed embedding $\Q\hookrightarrow\Q_\ell$). Then
    \begin{equation}\label{eq:relation-z0-a}
        \bar{\boldsymbol{z}}_0(\sigma_\ell)=(\loc_{\gl_n}f_{\bar{y}})(\tilde{\gs}_\ell)=-\tilde{\gs}_\ell\bar{a}=-\bar{a},
    \end{equation}
    using the properties described in \eqref{eq:definition-fy} and the fact that $\T_{m,s(\ell)}$ is unramified at $\ell$.

    According to point (b) of Lemma \ref{lem:compatibility-of-z}, there is an element $a\in\T_m$ such that
    \begin{equation*}
        \big(\Cor_{K_t[n]}^{K_t[n\ell]}(y)\big)(g)-T_\ell x(g)=(g-1)a
    \end{equation*}
    for every $g\in\Gal(\bar{\Q}/K_t[n])$. When quotienting modulo $p^{s(\ell)}$, the term $T_\ell x(g)$ vanishes (see Lemma \ref{lem:Tl-acts-as-0}) and hence the reduction of $a$ to $\T_{m,s(\ell)}$ coincides with $\bar{a}$. Localizing and restricting to $g=h_0\in \Gal(\bar{\Q}_\ell/K_t[n\ell]_{\gl_{n\ell}})=\Gal(\bar{\Q}_\ell/K[\ell]_{\gl_\ell'})$, we obtain the relation
    \begin{equation*}
        \sum_{i=0}^{\ell} \tilde{\gs}_\ell^i y(\tilde{\gs}_\ell^{-i} h_0\tilde{\gs}_\ell^i)-T_\ell x(h_0)=(h_0-1)a
    \end{equation*}
    as elements of $\T_m$. Since $\loc_{\gl_n}(\boldsymbol{y})\in H^1_{\ur}(K[\ell]_{\gl_\ell'},\T_m)$ by Proposition \ref{prop:z-in-the-Selmer}, all the elements appearing in the above sum coincide (see Lemma \ref{lem:app-unramified-cocycles}). Therefore, the previous relation becomes
    \begin{equation*}
        (\ell+1) y(h_0)-T_\ell x(h_0)=(h_0-1)a.
    \end{equation*}
    Choose now $h_0\in \Gal(\bar{\Q}_\ell/K[\ell]_{\gl_\ell'})$ with the property that its image to $\Gal(\Q_\ell^{\ur}/K_\gl)\cong \Gal(K[\ell]_{\gl_\ell'}^{\ur}/K[\ell]_{\gl_\ell'})$ coincides with $\Fr_\gl$. Then, combining the last equation with \eqref{eq:explicit-description-unramified-cocycles}, we obtain the formula
    \begin{equation*}
        (\ell+1)a_y-T_\ell a_x=(\Fr_\gl-1)(a-(\ell+1)b_y+T_\ell b_x)
    \end{equation*}
    on $\T_m$, where $a_x$, $b_x$, $a_y$, $b_y$ are elements of $\T_m$ uniquely determined by the cocycles $x$ and $y$ (see Appendix \ref{app:unramified-Galois-cohomology}). In particular, the classes of $a_x$ and $a_y$ in $\T_m/(\Fr_\gl-1)\T_m$ respectively coincide with the images of the classes $\loc_{\gl_{n}}\boldsymbol{x}$ and $\loc_{\gl_{n\ell}}\boldsymbol{y}$ under the isomorphisms
    \begin{equation}\label{eq:diagram-unramified-cohomology}
        \begin{tikzcd}
	{H^1_{\ur}(K_t[n]_{\gl_{n}},\T_m)} & {H^1(K_\gl^{\ur}/K_\gl,\T_m)} & {\T_m/(\Fr_\gl-1)\T_m} \\
	{H^1_{\ur}(K_t[n\ell]_{\gl_{n\ell}},\T_m)} & {H^1(K[\ell]_{\gl_{\ell}'}^{\ur}/K[\ell]_{\gl_{\ell}'},\T_m)} & {\T_m/(\Fr_{\gl_\ell'}-1)\T_m,}
	\arrow["{\inf^{-1}}", from=1-1, to=1-2]
	\arrow["\cong", from=1-2, to=1-3]
	\arrow["{\inf^{-1}}", from=2-1, to=2-2]
	\arrow["\cong", from=2-2, to=2-3]
	\arrow["\cong", from=2-3, to=1-3]
\end{tikzcd}
    \end{equation}
    where the horizontal unnamed maps are induced by evaluating cocycles at the Frobenius (see \cite[Lemma B.2.8]{Rubin00:euler-systems}). Since $s(\ell)\ge m$, we can apply the description of the action of the Frobenius on $\T_m$ observed in Lemma \ref{lem:charachteristic-polynomial-frobenius-Tm} to obtain
    \begin{equation}\label{eq:relation-middle-proof}
        (\ell+1)a_y-T_\ell a_x=((-1)^{\frac{k+j}{2}-1}T_\ell\Fr_\ell-(\ell+1))(a-p^{s(\ell)}c)
    \end{equation}
    for some $c\in\T_m$. Both terms are divisible by $p^{s(\ell)}$ and $\T_m$ is a (torsion-)free module over $\mathcal{R}_m$, therefore we can divide the above equality by $p^{s(\ell)}$. Moreover, as a consequence of Proposition \ref{prop:Eichler-Shimura-relation-big-Heegner-points}, we have the equality $a_y=\Fr_\ell a_x$ (see \cite[Lemma 2.32]{mastella-zerman2025:anticyclotomic} for more details). Therefore, the relation \eqref{eq:relation-middle-proof} implies that
    \begin{equation}\label{eq:key-formula}
        \bigg(\frac{\ell+1}{p^{s(\ell)}}\Fr_\ell-\frac{T_\ell}{p^{s(\ell)}}\bigg)\bar{a}_x=\bigg(\frac{(-1)^{\frac{k+j}{2}-1}T_\ell}{p^{s(\ell)}}\Fr_\ell-\frac{(\ell+1)}{p^{s(\ell)}}\bigg) \bar{a}
    \end{equation}
    on $\T_{m,s(\ell)}$, where $\bar{a}_x$ is the reduction of $a_x$ to $\T_{m,s(\ell)}$. 

    Notice that the isomorphisms of diagram \eqref{eq:diagram-unramified-cohomology} hold also with $\T_{m,s(\ell)}$ in place of $\T_m$. Since $\Fr_\gl$ acts trivially on $\T_{m,s(\ell)}$, we have that $\bar{a}_x$ is the element of $\T_{m,s(\ell)}$ attached to $\bar{\boldsymbol{x}}$ in this way, i.e.~$\bar{a}_x=\bar{x}(\Fr_{\gl_n})$. By \eqref{eq:relation-z0-a}, we also have that $\bar{a}=-f_{\bar{y}}(\tilde{\gs}_\ell)=-\bar{\boldsymbol{z}}_0(\sigma_\ell)$. Since the action of $\Gal(\bar{\Q}_\ell/K_\gl)$ is trivial on $\T_{m,s(\ell)}$, both these quantities don't depend on the chosen cocycles, and the second one doesn't depend on the chosen lift of $\sigma_\ell$. Therefore, we can rewrite the formula \eqref{eq:key-formula} as
    \begin{equation}\label{eq:formula-proof-Kt[n]}
        \bigg(\frac{(\ell+1)\Fr_\ell-T_\ell}{p^{s(\ell)}}\bigg)\tilde{\boldsymbol{\gk}}(n)_{t,m,s(\ell)}(\Fr_{\gl_n})=\bigg(\frac{(\ell+1)-(-1)^{\frac{k+j}{2}-1}T_\ell\Fr_\ell}{p^{s(\ell)}}\bigg) \tilde{\boldsymbol{\gk}}(n\ell)_{t,m,s(\ell)}(\tilde{\gs}_\ell)
    \end{equation}
    where $\tilde{\boldsymbol{\gk}}(n)_{t,m,s(\ell)}=D_n\mathfrak{z}(n)_{t,m,s(\ell)}=\bar{\boldsymbol{x}}$ and $\tilde{\boldsymbol{\gk}}(n\ell)_{t,m,s(\ell)}=(\res^{K_t[n\ell]}_{K[n]})^{-1}D_{n\ell}\mathfrak{z}(n\ell)_{t,m,s(\ell)}=[f_{\bar{y}}]$. By \cite[Lemma 1.2.1]{Mazur-Rubin:Kolyvagin-systems}, there are isomorphisms analogous to those defined in \eqref{eq:definition-finite-singular} for finite extensions of $K$ and primes above $\gl$. Using the obvious notation, we then have that
    \begin{equation*}
        \tilde{\boldsymbol{\gk}}(n)_{t,m,s(\ell)}(\Fr_{\gl_n})=\ga_{\gl_n}(\tilde{\boldsymbol{\gk}}(n)_{t,m,s(\ell)})\quad\text{and}\quad \tilde{\boldsymbol{\gk}}(n\ell)_{t,m,s(\ell)}(\tilde{\gs}_\ell)=\beta_{\gl_n}([\tilde{\boldsymbol{\gk}}(n\ell)_{t,m,s(\ell)}]_{\s}\otimes\sigma_\ell).
    \end{equation*}
    We now project the formula \eqref{eq:formula-proof-Kt[n]} to $\T_{m,s}$. Using the functoriality of the finite-singular decomposition, we can apply $\Cor^{K_t[1]}_{K_t}\circ (\res^{K_t[n]}_{K_t[1]})^{-1}$ and obtain that
    \begin{equation*}
        \bigg(\frac{(\ell+1)\Fr_\ell-T_\ell}{p^{s(\ell)}}\bigg)\ga_{\gl'}(\boldsymbol{\gk}'(n)_{t,m,s})=\bigg(\frac{(\ell+1)-(-1)^{\frac{k+j}{2}-1}T_\ell\Fr_\ell}{p^{s(\ell)}}\bigg) \beta_{\gl'}([\boldsymbol{\gk}'(n\ell)_{t,m,s}]_{\s}\otimes\sigma_\ell)
    \end{equation*}
    where $\gl'=\gl_n\cap K_t$. Since $\Fr_\ell^2=\Fr_\gl$ acts as the identity on $\T_{m,s}$, we can re-write the numerator of the coefficient of the right-hand side as $((\ell+1)\Fr_\ell-(-1)^{\frac{k+j}{2}-1}T_\ell)\Fr_\ell$. Thanks to Corollary \ref{cor:finer-choice-of-primes}, the coefficients of the left-hand side and the right-hand side are invertible endomorphisms of $\T_{m,s}$. Therefore, there is $\vartheta_{\gl'}\in\Aut(\T_{m,s})$ such that 
    \begin{equation*}
        \ga_{\gl'}(\boldsymbol{\gk}'(n)_{t,m,s})=\vartheta_{\gl'}(\beta_{\gl'}([\boldsymbol{\gk}'(n\ell)_{t,m,s}]_{\s}\otimes\sigma_\ell).
    \end{equation*}
    Since there were no constraints in the choice of the prime $\gl_n$ made at the beginning of the proof, the above relation holds for every prime $\gl'$ of $K_t$ above $\gl$. We then conclude by applying semi-local Shapiro's map to both sides.
\end{proof}

With this result in hand, we can conclude the proof of Theorem \ref{thm:main-thm}.

\begin{proof}[Proof of Theorem \ref{thm:main-thm}]
    To begin with, let's fix $m,s\in\Z_{>0}$ and $t\ge 0$. For every couple $(n,\ell)$ with $n\ell\in\mathcal{N}_{m,s}'$ and $\ell$ prime, define $\chi_{n,\ell}:=\vartheta_\ell^{-1}$ to be the inverse of the automorphisms that appear in Lemma \ref{lem:key-formula}. Define
    \begin{equation*}
        \boldsymbol{\gk}^{\ac}(n)_{t,m,s}:=\boldsymbol{\gk}(n)_{t,m,s}\otimes\bigotimes_{\ell'\mid n}\gs_{\ell'}\in H^1_{\mathcal{F}_{\Gr}}(K,\T_{t,m,s})\otimes\mathcal{G}(n).
    \end{equation*}
    Combining Theorem \ref{thm:kolyvagin-classes-lie-in-the-Selmer} with Lemma \ref{lem:key-formula}, we conclude that $\{\boldsymbol{\gk}^{\ac}(n)_{t,m,s}\}_{n\in\mathcal{N}_{t,m,s}'}$ is an element of $\KS(\T_{t,m,s},\mathcal{F}_{\Gr},\mathcal{P}_{m,s}',\{\chi_{n,\ell}\})$. 

    When varying $t,m,s$, we need to prove that these classes interpolate into a modified universal Kolyvagin system. First, the automorphisms $\chi_{n,\ell}$ don't depend on $t,m,s$. Indeed, fix $(n,\ell)$ such that $\ell$ is prime and $n\ell\in\mathcal{N}'$, and choose $t,t',m,m',s,s'$ with the property that $n\ell\in \mathcal{N}_{m,s}'\cap\mathcal{N}_{m',s'}'$. Then, it is clear from the proof of Lemma \ref{lem:key-formula} that the automorphisms $\vartheta_\ell=\chi_{n,\ell}^{-1}$ on $\T_{t,m,s}$ and $\T_{t',m',s'}$ both are the image of a unique automorphism $\vartheta_\ell$ of $\T^\ac$.

    Then, it is enough to check that, whenever $t'\ge t$, $m'\ge m$ and $s'\ge s$, the image of $\{\boldsymbol{\gk}^{\ac}(n)_{t',m',s'}\}_{n\in\mathcal{N}_{m',s'}}$ under the natural map
    \begin{equation*}
        \KS(\T_{t',m',s'},\mathcal{F}_{\Gr},\mathcal{P}_{m',s'},\{\chi_{n,\ell}\})\longrightarrow \KS(\T_{t,m,s},\mathcal{F}_{\Gr},\mathcal{P}_{m',s'},\{\chi_{n,\ell}\})
    \end{equation*}
    coincides with $\{\boldsymbol{\gk}^{\ac}(n)_{t,m,s}\}_{n\in\mathcal{N}_{m',s'}}$. When $t=t'$, this follows easily from the fact that the classes $\boldsymbol{\gk}^{\ac}(n)_{t',m',s'}$ and $\boldsymbol{\gk}^{\ac}(n)_{t,m,s}$ are built as the image of the same big Heegner class $\mathfrak{z}(n)_t$. When, instead, $m'=m$, $s'=s$ and $t'>t$, one uses the compatibility of Lemma \ref{lem:vertical-compatibility-kn}.

    Lastly, it is clear from the definitions that $\boldsymbol{\gk}^{\ac}(1)_{t,m,s}=\boldsymbol{\gk}(1)_{t,m,s}=(\sh_t\circ \Cor^{K_t[1]}_{K_t})(\mathfrak{z}(1)_{t,m,s})$. Taking the inverse limit on $t,m,s$, we obtain the relation claimed in Theorem \ref{thm:main-thm}.
\end{proof}

\section{Anticyclotomic Iwasawa theory}\label{sec:Anticyclotomic-Iwasawa-theory}

In this section we recover a classical consequence of the existence of a nontrivial Kolyvagin system to the Iwasawa theory of the representation $\T^\dag$. In particular, following the ideas of \cite{Fouquet13:Dihedral}, we give a proof of one divisibility of the big Heegner point main conjecture in this context (see \cite[Conjecture 10.8]{Longo-Vigni11:quaternion-algebras-Hida-families}). From now on, we will work under the following assumption.

\begin{assumption}\label{ass:R-regular}
    The ring $\calR$ is regular.
\end{assumption}

\subsection{$\calR^\ac$-modules} 

We start by describing in detail the algebraic structure of $\calR$ and $\calR^{\ac}$.

\begin{lemma}\label{lem:regularity}
    The rings $\calR$ and $\calR^\ac$ are complete Noetherian regular local integrally closed UFDs, whose height 1 prime ideals are principal.
\end{lemma}
\begin{proof}
    Combining Lemma \ref{lem:algebraic-properties-of-R} with Assumption \ref{ass:R-regular}, we have that $\mathcal{R}$ is a complete Noetherian regular local ring. By \cite[Theorem 3.3, Exercise 8.6, Theorem 19.5]{Matsumura89:Commutative-ring-theory}, the same is true for $\calR^\ac$. Then, by Auslander-Buchsbaum's theorem (see \cite[Theorem 20.3]{Matsumura89:Commutative-ring-theory}), both $\mathcal{R}$ and $\mathcal{R}^{\ac}$ are UFDs. Therefore, they are also integrally closed (see \cite[Example 9.1]{Matsumura89:Commutative-ring-theory}) and every high $1$ prime ideal is principal by \cite[Theorem 20.1]{Matsumura89:Commutative-ring-theory}.
\end{proof}

There is a rich theory of pseudo-isomorphisms between finitely generated  $\calR^\ac$-modules (see e.g.~\cite[§VII.4.4]{Bourbaki98:commutative-algebra}). For what is needed here, let's recall that the \emph{characteristic ideal} of a finitely generated torsion $\calR^\ac$-module $M$ is
\begin{equation*}
         \Char_{\calR^{\ac}}(M)=\prod_{\hight\p=1} \p^{\length_{\calR^\ac_\p}(M_\p)},
\end{equation*}
where the product runs over the set of high-one primes of $\calR^\ac$. This is a principal ideal of $\calR^\ac$ with the property that $M$ is pseudo-isomorphic to $\calR^\ac/\Char_{\calR^{\ac}}(M)$.

\subsection{Specializations}

We shall now study the behaviour of Greenberg (and other) Selmer modules when applying some families of specialization maps, in the following sense.

\begin{definition}
    Let $R$ be a complete local Noetherian $\OF$-algebra with finite residue field of characteristic $p$.
    \begin{itemize}
        \item[(i)] An $\OF$-algebra morphism $\sfrak\colon R\to S$, where $S$ is a complete local Noetherian domain with finite residue field of characteristic $p$, is called an \emph{S-specialization} of $R$.
        \item[(ii)] If $T$ is an $R$-module, denote by $T_\sfrak$ (or $T_S$, when the map $\sfrak$ is clear from the context) the $S$-module $T\otimes_{\sfrak}S$.
    \end{itemize}
\end{definition}

Notice that we allow the case $\sfrak=\id\colon R\to R$. When $\sfrak\colon \calR^{\ac}\to S$ is a specialization of $\calR^\ac$, define $\V_{\mathfrak{s}}^{\ac}$, $\A_{\mathfrak{s}}^{\ac}$, $F_v^{\pm}(\T_{\mathfrak{s}}^{\ac})$, $F_v^{\pm}(\V_{\mathfrak{s}}^{\ac})$ and $F_v^{\pm}(\A_{\mathfrak{s}}^{\ac})$ to be $\T_{\mathfrak{s}}^\ac\otimes_S\Frac(S)$, $\T_{\mathfrak{s}}^\ac\otimes_S S^\vee$, $F_v^{\pm}(\T^\ac)\otimes_{\mathfrak{s}}S$, $F_v^{\pm}(\T_{\mathfrak{s}}^\ac)\otimes_S\Frac(S)$ and $F_v^{\pm}(\T_{\mathfrak{s}}^\ac)\otimes_S S^\vee$ respectively, where $(\bullet)^\vee:=\Hom_{\cont}(\bullet,\Q_p/\Z_p)$ denotes the Pontryagin dual of $\bullet$. 

One can define the Greenberg Selmer structure on $\V^\ac:=\V^\ac_{\id}$ and on $\A^\ac:=\A^\ac_{\id}$ exactly as in Definition \ref{dfn:strict-Greenberg-selmer-structure}. For any specialization $\sfrak$, this induces a Selmer structure on $\T_{\mathfrak{s}}^{\ac}$, $\V_{\mathfrak{s}}^{\ac}$ and $\A_{\mathfrak{s}}^{\ac}$ by propagation, to be denoted again by $\mathcal{F}_\Gr$. However, in this cases we also have the following natural Selmer structure.

\begin{definition}
        Let $L$ be a finite extension of $K$, $\sfrak\colon \calR^{\ac}\to S$ be a specialization of $\calR^\ac$ and let $X_\sfrak$ be any of $\T_{\mathfrak{s}}^{\ac}$, $\V_{\mathfrak{s}}^{\ac}$ and $\A_{\mathfrak{s}}^{\ac}$. The Selmer structure $\mathcal{F}_{\sfrak}$ on $X_\sfrak$ over $L$ is defined by setting 
        \begin{equation*}
            H^1_{\mathcal{F}_\sfrak}(L_w,X_\sfrak):=\begin{cases}
                H_{\ur}^1(L_w,X_\sfrak)\quad &\text{if $w\nmid p$}\\
                \ker\big(H^1(L_w,X_\sfrak)\to H^1(L_w,F_w^-(X_\sfrak))\big)\quad &\text{if $w\mid p$}
            \end{cases}
        \end{equation*}
        where $w$ runs over all places of $L$.
    \end{definition}

When $\sfrak=\id$ we have that $\mathcal{F}_\Gr=\mathcal{F}_{\sfrak}$, but in general they are different. Notice that the argument of Lemma \ref{lem:easier-description-Greenberg-condition} doesn't work here because $\sfrak$ is not necessarily surjective. However, for every specialization $\sfrak\colon \calR^{\ac}\to S$ there is a natural map
\begin{equation}\label{eq:Greenberg-and-sfrak-selmer-structure}
    H^1_{\mathcal{F}_\Gr}(L,X_{\id})\longrightarrow H^1_{\mathcal{F}_\sfrak}(L,X_\sfrak).
\end{equation}

\subsection{Comparison with other Selmer groups}

In this section, we exploit the link between the $\mathcal{F}_\sfrak$-Selmer groups and other Selmer groups, such as Nekov\' a\v r's extended Selmer groups and Bloch--Kato Selmer groups. This is needed for an explicit comparison with other works, such as Howard \cite{Howard04:heegner-point-kolyvagin-system} and Fouquet \cite{Fouquet13:Dihedral}.

Let $X$ be any of $\T_{\mathfrak{s}}^{\ac}$, $\V_{\mathfrak{s}}^{\ac}$ and $\A_{\mathfrak{s}}^{\ac}$. In this setting, Nekov\' a\v r defined some \emph{extended Selmer groups} $\tilde{H}^1_{\f}(K,X)$ in terms of his Selmer complexes. These Selmer groups are linked to the $\mathcal{F}_\sfrak$-Selmer groups via the exact sequence (see \cite[Lemma 9.6.3]{Nekovar06:selmer-complexes})
\begin{equation}\label{eq:Nekovar-exact-sequence}
    \bigoplus_{v\mid p} H^0(K_v, F_v^-(X))\to \tilde{H}^1_{\f}(K,X)\to H^1_{\mathcal{F}_\sfrak}(K,X)\to 0.
\end{equation}
We show that this exact sequence, combined with Assumption \ref{ass:H.stz}, gives an isomorphism between the two Selmer groups in some relevant cases.

\begin{lemma}\label{lem:equality-Selmer-Nekovar}
    Let $X$ be one of $\T_\sfrak^\ac$ or $\A_\sfrak^\ac$. Then
    \begin{itemize}
        \item[(i)] $H^0(K_v, F_v^-(X))=\{0\}$ for every $v\mid p$;
        \item[(ii)] $\tilde{H}^1_{\f}(K,X)\cong H^1_{\mathcal{F}_\sfrak}(K,X)$.
    \end{itemize}
\end{lemma}
\begin{proof}
    Let $X=\T_\sfrak^\ac$. Since the residue field of $S$ is finite, the residual representation of $F_v^-(\T_\sfrak^\ac)=F_v^-(\T^\ac)\otimes_\sfrak S$ is a finite base change of the 1-dimensional representation $F_v^-(\bar{\T}^\ac)$, that has no $G_{K_v}$-invariants by Assumption \ref{ass:H.stz}. Therefore, $H^0(K_v,F_v^-(\T_\sfrak^\ac)/\m_S)=\{0\}$, and we obtain point (i) by applying Lemma \ref{lem:nakayama-for-Galois-representation}. Point (ii) descends from (i) by the exact sequence \eqref{eq:Nekovar-exact-sequence}. 
    
    Let now $M:=H^0(K_v, F_v^-(\A_\sfrak^\ac))^\vee$ be the Pontryagin dual of $H^0(K_v, F_v^-(\A_\sfrak^\ac))$. By Nekov\' a\v r's duality theory (see the first paragraphs of \cite[§2]{Kim-Longo:anticyclotomic-main-conjecture} for a compact treatment), $M$ is isomorphic to the module of $G_{K_v}$-coinvariants of $\Hom_\calR(F_v^-(\T_\sfrak^\ac),S)$, that is a finitely generated $S$-module. Therefore, by Nakayama's lemma, in order to show (i) it is enough to show that $M/\m_{S}M=\{0\}$. For this, taking again Pontryagin duals, we obtain that
    \begin{equation*}
        (M/\m_{S}M)^\vee\cong M^\vee[\m_S]=H^0(K_v,F_v^-(\A_\sfrak^\ac))[\m_S]=H^0(K_v,F_v^-(\A_\sfrak^\ac)[\m_S]).
    \end{equation*}
    Now, the Galois representation $F_v^-(\A_\sfrak^\ac)[\m_S]$ is isomorphic to $F_v^-(\T_\sfrak^\ac)/\m_S F_v^-(\T_\sfrak^\ac)$, which was shown to have no $G_{K_v}$-invariants in the first part of the proof, therefore we obtain point (i). Point (ii) follows again from the exact sequence \eqref{eq:Nekovar-exact-sequence}.
\end{proof}

When $S$ is a discrete valuation ring (DVR) finite over $\OF$, there is the notion of the Bloch-Kato condition on $\T_{\mathfrak{s}}^{\ac}$ and $\A_{\mathfrak{s}}^{\ac}$, which we now recall following \cite[Definition 5.20]{Fouquet13:Dihedral}. To begin with, notice that the natural maps $\iota\colon \T^\ac_\sfrak\to\V_\sfrak^\ac$ and $\eta\colon\V_\sfrak^\ac\to\A_\sfrak^\ac$ extend to maps in cohomology, to be denoted with the same name.

\begin{definition}
    Let $L$ be a finite extension of $K$ and $\mathfrak{s}$ be a specialization of $\mathcal{R}^{\ac}$ to a discrete valuation ring $S$ finite over $\OF$. Define the \emph{Bloch--Kato Selmer structure $\mathcal{F}_{\BK}$} on $\T_{\mathfrak{s}}^\ac$ by setting the local conditions
    \begin{equation*}
            H^1_{\mathcal{F}_{\BK}}\big(L_w,\T_{\mathfrak{s}}^\ac\big)=\begin{cases}
                \iota^{-1}\big(H_{\ur}^1(L_w,\V_{\mathfrak{s}}^\ac)\big)\quad &\text{if $w\mid N$}\\
                \iota^{-1}\big(H^1(L_w,F_v^+(\V_{\mathfrak{s}}^\ac))\big)\quad &\text{if $w\mid p$}\\
                H_{\ur}^1(L_w, \T_{\mathfrak{s}}^\ac) &\text{if $w\nmid Np$}.
            \end{cases}
        \end{equation*}
    where $w$ varies among the places of $L$ and $v=w\cap K$. Similarly, we define the \emph{Bloch--Kato Selmer structure $\mathcal{F}_{\BK}$} on $\A_{\mathfrak{s}}^\ac$ by setting local conditions
    \begin{equation*}
            H^1_{\mathcal{F}_{\BK}}(L_w,\A_{\mathfrak{s}}^\ac)=\begin{cases}
                \eta\big(H_{\ur}^1(L_w,\V_{\mathfrak{s}}^\ac)\big)\quad &\text{if $w\mid N$}\\
                \eta\big(H^1(L_w,F_v^+(\V_{\mathfrak{s}}^\ac))\big)\quad &\text{if $w\mid p$}\\
                H_{\ur}^1(L_w, \A_{\mathfrak{s}}^\ac) &\text{if $w\nmid Np$}.
            \end{cases}
        \end{equation*}
\end{definition}

In general, thanks to \cite[Lemma 3.5]{Rubin00:euler-systems}, there are inequalities $H^1_{\mathcal{F}_\sfrak}(L,\T_\sfrak^\ac)\subseteq H^1_{\mathcal{F}_\BK}(L,\T_\sfrak^\ac)$ and $H^1_{\mathcal{F}_\sfrak}(L,\A_\sfrak^\ac)\supseteq H^1_{\mathcal{F}_\BK}(L,\A_\sfrak^\ac)$. However, thanks to the assumptions of Section \ref{sec:controlling-ramification}, we have the following stronger result when $L=K$.

\begin{lemma}\label{lem:Greenberg=Bloch-Kato}
    Let $\sfrak\colon \calR^{\ac}\to S$ be a specialization to a discrete valuation ring $S$ finite over $\OF$. Then
    \begin{equation*}
        H^1_{\mathcal{F}_\sfrak}(K,\T_\sfrak^\ac)= H^1_{\mathcal{F}_\BK}(K,\T_\sfrak^\ac)\quad\text{and}\quad H^1_{\mathcal{F}_\sfrak}(K,\A_\sfrak^\ac)= H^1_{\mathcal{F}_\BK}(K,\A_\sfrak^\ac).
    \end{equation*}
\end{lemma}
\begin{proof}
    If $v$ is a place of $K$ not dividing $Np$, by \cite[Lemma 3.5]{Rubin00:euler-systems} the $\mathcal{F}_\sfrak$ and the $\mathcal{F}_{\BK}$ local conditions coincide at $v$. If $v\mid N$, by considering the $\mathfrak{I}_v$-cohomology of the exact sequence $0\to \T^\ac_\sfrak\to\V^\ac_\sfrak\to\A^\ac_\sfrak\to 0$ and noting that the image of $(\V^\ac_\sfrak)^{\mathfrak{I}_v}$ in $(\A^\ac_\sfrak)^{\mathfrak{I}_v}$ is the maximal divisible subgroup $(\A^\ac_\sfrak)^{\mathfrak{I}_v}_{\di}$, we obtain that $(\A^\ac_\sfrak)^{\mathfrak{I}_v}/(\A^\ac_\sfrak)^{\mathfrak{I}_v}_{\di}$ injects into $H^1(\mathfrak{I}_v, \T^\ac_\sfrak)$, which is torsion-free thanks to Proposition \ref{prop:tamagawa-numbers-at-N}. Therefore, by point (iii) of \cite[Lemma 3.5]{Rubin00:euler-systems}, we conclude that $H^1_{\mathcal{F}_\Gr}(K_v,\T_\sfrak^\ac)= H^1_{\mathcal{F}_\BK}(K_v,\T_\sfrak^\ac)$ and $H^1_{\mathcal{F}_\Gr}(K_v,\A_\sfrak^\ac)= H^1_{\mathcal{F}_\BK}(K_v,\A_\sfrak^\ac)$.

    Assume now that $v\mid p$. By the definition of the corresponding local conditions, it is enough to show that the map $H^1(K_v, F_v^-(\T^\ac_\sfrak))\to H^1(K_v, F_v^-(\V^\ac_\sfrak))$ is injective and that the map $H^1(K_v, F_v^+(\V^\ac_\sfrak))\to H^1(K_v, F_v^+(\A^\ac_\sfrak))$ is surjective. The kernel of the first map is $H^0(K_v, F_v^-(\A^\ac_\sfrak))$, which is trivial by Lemma \ref{lem:equality-Selmer-Nekovar}. The cokernel of the second map is $H^2(K_v, F_v^+(\T^\ac_\sfrak))$, which is isomorphic to $H^0(K_v,F_v^-(\T^\ac_\sfrak))$ by duality (see the proof of Proposition \ref{lem:easier-description-Greenberg-condition}), which is trivial by Lemma \ref{lem:equality-Selmer-Nekovar}.
\end{proof}

\subsection{The big Heegner point main conjecture}\label{sec:big-heegner-point-main-conjecture}

Following the arguments of \cite[§5-6]{Fouquet13:Dihedral}, we can now use the universal modified Kolyvagin system of Theorem \ref{thm:main-thm} to prove one divisibility of the big Heegner point main conjecture. The main idea is to use specializations of $\calR^\ac$ to discrete valuation rings.

Let $\sfrak\colon \calR^\ac\to S$ be a specialization to a DVR finite over $\OF$. For every $i\in\Z_{>0}$, let $\mathcal{Q}_i$ be the set of primes $\ell\in\mathcal{P}$ whose Frobenius is conjugated with $\tau_c$ in $\Gal(K(\T^\ac_\sfrak/p^i)/\Q)$, choose a subset $\mathcal{Q}'$ of $\mathcal{Q}_1$ and call $\mathcal{Q}_i':=\mathcal{Q}'\cap\mathcal{Q}_i$. Denote by $\mathcal{M}'$ the set of square-free products of elements of $\mathcal{Q}'$.  For every couple $(n,\ell)$ with $\ell$ prime and $n\ell\in\mathcal{M}'$, let $\chi_{n,\ell}\in\Aut_{S}(\T^{\ac}_\sfrak)$ and denote with the same letter the induced automorphism of $\T^{\ac}_\sfrak/p^i$, for every $i\in\Z_{>0}$. Then, exactly as done in Section \ref{sec:modified-Kolyvagin-systems} (see \cite[Section 2.7]{mastella-zerman2025:anticyclotomic} for the general theory), one can define the $S$-module of universal Kolyvagin systems for $(\T^\ac_\sfrak,\mathcal{F}_\sfrak,\mathcal{Q}',\{\chi_{n,\ell}\})$ as
\begin{equation*}
            \overline{\KS}(\T^{\ac}_\sfrak,\mathcal{F}_\sfrak,\mathcal{Q}', \{\chi_{n,\ell}\}):=\varprojlim_{i}\varinjlim_{j\ge i} \KS(\T^\ac_\sfrak/p^i,\mathcal{F}_\sfrak,\mathcal{Q}_{j}', \{\chi_{n,\ell}\}).
\end{equation*}
If $\{\chi_{n,\ell}\}$ is a set of automorphisms of $\T^\ac$ as in Section \ref{sec:modified-Kolyvagin-systems} and $\mathcal{P}'$ is the set of primes of Definition \ref{dfn:kolyvagin-primes}, the functoriality of Kolyvagin systems with respect to the specialization $\sfrak$ together with equation \eqref{eq:Greenberg-and-sfrak-selmer-structure} yield a canonical map
\begin{equation}\label{eq:specialization-kolyvagin-system}
    \sfrak\colon\overline{\KS}(\T^\ac,\mathcal{F}_\Gr,\mathcal{P}',\{\chi_{n,\ell}\})\longrightarrow \overline{\KS}(\T_\sfrak^\ac,\mathcal{F}_\sfrak,\mathcal{P}',\{\chi_{n,\ell}\}),
\end{equation}
where we denote in the same way the automorphism of $\T_\sfrak^\ac$ induced by $\chi_{n,\ell}$. 

\begin{lemma}\label{lem:Howard-result}
    Let $\sfrak\colon \calR^{\ac}\to S$ be a specialization to a discrete valuation ring finite over $\OF$ and let $\boldsymbol{\gk}_\sfrak\in \overline{\KS}(\T^\ac_\sfrak,\mathcal{F}_\sfrak,\mathcal{P}',\{\chi_{n,\ell}\})$ with the property that $\boldsymbol{\gk}_\sfrak(1)\ne 0$. Then $H^1_{\mathcal{F}_\sfrak}(K,\T_{\sfrak}^\ac)$ is free of rank $1$ and there exists a torsion module $M$ finite over $S$ with
    \begin{equation*}
        \length_S(M)\le \length_S\big(H^1_{\mathcal{F}_\sfrak}(K,\T_{\sfrak}^\ac)/S\cdot\boldsymbol{\gk}_\sfrak(1)\big)
    \end{equation*}
    such that $H^1_{\mathcal{F}_\sfrak}(K,\A_{\sfrak}^\ac)\cong \Frac(S)/S\oplus M\oplus M$.
\end{lemma}
\begin{proof}
   First notice that, if we show that $\T_\sfrak^\ac$ and $\mathcal{F}_\sfrak$ satisfy Hypotheses H.0-H.5 of \cite[§1.3]{Howard04:heegner-point-kolyvagin-system}, we can apply Proposition \ref{prop:howard-theorem-revised} -- that is a variation of \cite[Theorem 1.6.1]{Howard04:heegner-point-kolyvagin-system} -- in order to conclude. Therefore, we are left to show that Hypotheses H.0-H.5 of \textit{loc.~cit.}~hold.

    Hypothesis H.0 is true since $\T^\ac$ is free of rank 2 over $\calR^\ac$. Since $S$ has finite residue field, the residual representation $\bar{\T}_\sfrak^\ac:=\T_\sfrak^\ac/\m_S\T_\sfrak^\ac$ is equivalent to $\T^\ac/\m_{\calR^\ac}\T^\ac$ up to finite base change, which is absolutely irreducible by Remark \ref{rk:irreducibility-over-ring-class-fields}, yielding Hypothesis H.1.

    By Lemma \ref{lem:Greenberg=Bloch-Kato}, we know that $\mathcal{F}_\sfrak=\mathcal{F}_\BK$ on $\T_\sfrak^\ac$, therefore all local conditions are induced by local conditions on $\V^\ac_\sfrak$. This implies that, for every prime $v$ of $K$, the quotient of $H^1(K_v,\T_\sfrak^\ac)$ with $H^1_{\mathcal{F}_{\sfrak}}(K_v,\T_\sfrak^\ac)$ is $\calR_\sfrak^\ac$-torsion free. Therefore, the straightforward generalization of \cite[Lemma 3.7.1]{Mazur-Rubin:Kolyvagin-systems} to the anticyclotomic case implies that \cite[Hypothesis H.3]{Howard04:heegner-point-kolyvagin-system} is satisfied.

    Hypotheses H.4-H.5 of \textit{loc.~cit.}~are a consequence of the existence of the pairing \eqref{eq:pairing}, which induces a perfect, Galois invariant pairing 
    \begin{equation*}
        \T_\sfrak^\ac\times\T_\sfrak^\ac\to S(1),
    \end{equation*}
    exactly as shown in the proof of \cite[Proposition 2.1.3]{Howard04:heegner-point-kolyvagin-system}.
    
    Let now $F=K_\infty(\A^\ac_\sfrak)$, so that $G_F$ acts trivially on $\T^\ac_\sfrak=\varprojlim_j\A_\sfrak^\ac[\mathfrak{m}_S^j]$. This, together with the existence of the pairing above, implies that $F$ contains all the $p$-power roots of unit. Thanks to our big image assumption (Assumption \ref{ass:big-image}), the image of $G_K$ in $\Aut(\bar{\T}_\sfrak^\ac)$ contains a nontrivial scalar matrix. Therefore, we can apply Sah's lemma \cite[Theorem 2.7.(c)]{sah:1968automorphisms} and conclude that $H^1(F/K,\bar{\T}_\sfrak^\ac)=0$, as required by \cite[Hypothesis H.2]{Howard04:heegner-point-kolyvagin-system}.
\end{proof}

\begin{theorem}\label{thm:iwasawa-main-conjecture}
    Assume that the system $\boldsymbol{\gk}^\ac\in \overline{\KS}(\T^\ac,\mathcal{F}_\Gr,\mathcal{P}',\{\chi_{n,\ell}\})$ of Theorem \ref{thm:main-thm} is such that $\boldsymbol{\gk}^{\ac}(1)\ne 0$. Then  
    \begin{itemize}
        \item[(a)] $H^1_{\mathcal{F}_\Gr}(K,\T^\ac)$ is a torsion-free  $\calR^\ac$-module of rank 1.
        \item[(b)] $\Char_{\calR^\ac}(H^1_{\mathcal{F}_\Gr}\big(K,\A^\ac)^\vee_{\tors}\big)\supseteq \Char_{\calR^\ac}\big(H^1_{\mathcal{F}_\Gr}(K,\T^{\ac})/\calR^\ac\cdot \boldsymbol{\gk}^\ac(1)\big)^2$.
    \end{itemize}
\end{theorem}
\begin{proof}
    The proof of this fact is a direct consequence of the arguments of \cite[§6]{Fouquet13:Dihedral} (see in particular the proof of \cite[Theorem 6.3]{Fouquet13:Dihedral}). We explain in detail how our setting fits in Fouquet's theory.

    (a) By Lemma \ref{lem:equality-Selmer-Nekovar}, we have the equality $H^1_{\mathcal{F}_\Gr}(K,\T^\ac)\cong \tilde{H}^1_\f(K,\T^\ac)$. Therefore, the torsion-freeness of $H^1_{\mathcal{F}_\Gr}(K,\T^\ac)$ is a consequence of \cite[Proposition 6.2.(i)]{Fouquet13:Dihedral} (or of \cite[§1.3.3]{Perrin-Riou:p-adic-L-functions-and-p-adic-representations}, as explained in \cite[Lemma 3.3]{Castella-Wan:iwasawa-main-conjecture-for-GL2}). Moreover, the proof of \cite[Proposition 6.5]{Fouquet13:Dihedral} goes along with Corollary 5.21 of \textit{loc.~cit.}~replaced by our Lemma \ref{lem:Howard-result}, giving the rank-1 result. Notice that the number $r$ that appears in the statement of \cite[Proposition 6.2.(i)]{Fouquet13:Dihedral} is zero thanks to Assumption \ref{ass:H.stz}.

    (b) Assume that $\Char_{\calR^\ac}(H^1_{\mathcal{F}_\Gr}(K,\T^\ac)/\boldsymbol{\gk}^\ac(1))^2\not\subseteq\Char_{\calR^\ac}(H^1_{\mathcal{F}_\Gr}(K,\A^\ac)^\vee_{\tors})$. Exactly as in the proof of \cite[Theorem 6.3]{Fouquet13:Dihedral}, by applying some density results one finds a specialization $\sfrak\colon\calR^\ac\to S$ to a DVR finite over $\OF$ such that $\boldsymbol{\gk}^\ac_\sfrak(1)\ne 0$ and 
    \begin{equation*}
        \length_S\big(\tilde{H}_\f^1(K,\A^\ac_\sfrak)^\vee_{\tors}\big)> 2\cdot\length_S\big(\tilde{H}^1_{\f}(K,\T^\ac_\sfrak)/\boldsymbol{\gk}_\sfrak^\ac(1)\big),
    \end{equation*}
    where $\boldsymbol{\gk}^\ac_\sfrak$ is the image of $\boldsymbol{\gk}^\ac$ with respect to \eqref{eq:specialization-kolyvagin-system}. However, as a consequence of the inequality of Lemma \ref{lem:Howard-result}, we have that
     \begin{equation*}
        \length_S\big(H_{\mathcal{F}_\sfrak}^1(K,\A^\ac_\sfrak)^\vee_{\tors}\big)\le 2\cdot\length_S\big(H_{\mathcal{F}_\sfrak}^1(K,\T^\ac_\sfrak)/\boldsymbol{\gk}_\sfrak^\ac(1)\big).
    \end{equation*}
    This is a contradiction, as $\tilde{H}_\f^1(K,\A^\ac_\sfrak)\cong H_{\mathcal{F}_\sfrak}^1(K,\A^\ac_\sfrak)$ and $\tilde{H}^1_{\f}(K,\T^\ac_\sfrak)\cong H_{\mathcal{F}_\sfrak}^1(K,\T^\ac_\sfrak)$ by Lemma \ref{lem:equality-Selmer-Nekovar}.
\end{proof}

\begin{remark}
    If we assume that $p\nmid\varphi(N)$, the work of Cornut--Vatsal \cite{cornut-vatsal2007:nontriviality} (see also \cite[Corollary 3.1.2]{Howard07:variation-of-Heegner-points-in-Hida-families} and the last lines of \cite[§4.3]{Castella-Wan:iwasawa-main-conjecture-for-GL2}) implies that $\boldsymbol{\gk}^\ac(1)\ne 0$, therefore the above theorem is unconditional in this case. Using a different method, the big Heegner point main conjecture has been recently proven by Castella--Wan (see \cite[Theorem 5.1]{Castella-Wan:iwasawa-main-conjecture-for-GL2}) under some mild hypotheses.
\end{remark}

\appendix

\section{Explicit unramified Galois cohomology}\label{app:unramified-Galois-cohomology}

Let $\ell\ne p$ be two odd primes and let $T$ be an unramified $\Z_p\llbracket \Gal(\bar{\Q}_\ell/\Q_\ell)\rrbracket$-module which is finitely generated as a $\Z_p$-module. Let $F$ be a finite extension of the local field $\Q_\ell$ and let  $\Fr_F\in\Gal(F^{\ur}/F)$ be the Frobenius element, that is a profinite generator of $\Gal(F^{\ur}/F)\cong\hat{\Z}$. Since $T$ is unramified, the action of $\Fr_F$ on $T$ is well defined. Denote by
\begin{equation*}
    \pi_{\ur}\colon \Gal(\bar{\Q}_\ell/F)\twoheadrightarrow \Gal(F^{\ur}/F)
\end{equation*}
the natural projection. Then, to any $g\in \Gal(\bar{\Q}_\ell/F)$ we can attach a unique exponent $u(g)\in\hat{\Z}$ such that $\pi_{\ur}(g)=\Fr_F^{u(g)}$.

Let $x\in Z^1(F,T)$ be a cocycle whose cohomology class $[x]$ lies in $H^1_{\ur}(F,T)$. Then, the cocycle $x$ is inflated by a cocycle $\tilde{x}\in Z^1(F^{\ur}/F,T)$, up to summing a coboundary. Then, for every $g\in \Gal(\bar{\Q}_\ell/F)$ such that $u(g)\in\Z_{>0}$ there is $b_x\in T$ such that
\begin{equation}\label{eq:explicit-description-unramified-cocycles}
    x(g)=\tilde{x}(\Fr_F^{u(g)})+(g-1)b_x=(1+\Fr_F+\dots+\Fr_F^{u(g)-1})a_x+(\Fr_F^{u(g)}-1)b_x,
\end{equation}
where $a_x:=\tilde{x}(\Fr_F)\in T$. In fact, the first equality holds for every $g\in \Gal(\bar{\Q}_\ell/F)$, while the requirement $u(g)\in\Z_{>0}$ is needed only for the second equality. Via the isomorphism 
\begin{equation*}
    H^1(F^{\ur}/F,T)\overset{\cong}{\longrightarrow}T/(\Fr_F-1)T
\end{equation*}
induced by evaluating cocycles at $\Fr_F$ (see \cite[Lemma B.2.8]{Rubin00:euler-systems}), the class $[x]$ corresponds to the class of $a_x$ in the codomain.

\begin{lemma}\label{lem:app-unramified-cocycles}
    Let $x\in Z^1(F,T)$ be an unramified cocycle as above and let $\tau\in \Gal(\bar{\Q}_\ell/F^{\ur})$. Then $\gt x=x$.
\end{lemma}
\begin{proof}
    We need to prove that $\gt x(\gt^{-1} g\gt)=x(g)$ for every $g\in \Gal(\bar{\Q}_\ell/F)$. First, notice that $\gt x(\gt^{-1} g\gt)=x(\gt^{-1} g\gt)$, as $T$ is unramified. Then, by the explicit description of \eqref{eq:explicit-description-unramified-cocycles}, it is enough to show that $u(g)=u(\gt^{-1} g\gt)$, i.e.~that $\pi_{\ur}(g)=\pi_{\ur}(\gt^{-1} g\gt)$. The relation $\gt^{-1} g\gt=\gt^{-1} g\gt g^{-1} g$ together with the normality of $\Gal(\bar{\Q}_\ell/F^{\ur})$ in $\Gal(\bar{\Q}_\ell/F)$ yields that $\pi_{\ur}(g)=\pi_{\ur}(\gt^{-1} g\gt)$, as desired.
\end{proof}

\section{Howard's method revised}\label{sec:howards-method-revised}

The purpose of this appendix is to cover a technical step in the proof of Lemma \ref{lem:Howard-result}. In particular, we need to prove an analogue of \cite[Theorem 1.6.1]{Howard04:heegner-point-kolyvagin-system} in our setting. The arguments presented here are an adaptation of those of \cite[§2.12]{mastella-zerman2025:anticyclotomic}, which in turn are profoundly based on those of \cite[§1.6]{Howard04:heegner-point-kolyvagin-system}.

Following the notation of Section \ref{sec:big-heegner-point-main-conjecture}, let $\sfrak\colon \calR^{\ac}\to S$ be a specialization to a discrete valuation ring finite over $\OF$. Let $\mathcal{F}$ be a Selmer structure on $\T^\ac_\sfrak$ and let $\{\chi_{n,\ell}\}$ be a set of automorphisms of $\T^\ac_\sfrak$, indexed on the set of couples $(n,\ell)$ with $\ell$ prime and $n\ell\in\mathcal{N}'$. Here, $\mathcal{N}'$ is the set of square-free products of elements of the set $\mathcal{P}'$ introduced in Definition \ref{dfn:kolyvagin-primes}. For every $i\in\Z_{>0}$, denote by $\mathcal{P}'_i$ the set of primes $\ell\in\mathcal{P}'$ whose Frobenius is conjugated with $\tau_c$ in $\Gal(K(\T^\ac_\sfrak/p^i)/\Q)$. The main result of this section is the following proposition.

\begin{proposition}\label{prop:howard-theorem-revised}
    Suppose that the triple $(\T^\ac_\sfrak,\mathcal{F},\mathcal{P}')$ satisfies Hypotheses H.0--H.5 of \cite{Howard04:heegner-point-kolyvagin-system} and assume that there is $\boldsymbol{\gk}_\sfrak\in \overline{\KS}(\T^\ac_\sfrak,\mathcal{F},\mathcal{P}',\{\chi_{n,\ell}\})$ with the property that $\boldsymbol{\gk}_\sfrak(1)\ne 0$. Then $H^1_{\mathcal{F}}(K,\T_{\sfrak}^\ac)$ is free of rank $1$ and there exists a torsion module $M$ finite over $S$ with
    \begin{equation*}
        \length_S(M)\le \length_S\big(H^1_{\mathcal{F}}(K,\T_{\sfrak}^\ac)/S\cdot\boldsymbol{\gk}_\sfrak(1)\big)
    \end{equation*}
    such that $H^1_{\mathcal{F}}(K,\A_{\sfrak}^\ac)\cong \Frac(S)/S\oplus M\oplus M$.
\end{proposition}
\begin{proof}
    The proof follows word by word the proof of \cite[Theorem 1.6.1]{Howard04:heegner-point-kolyvagin-system}, with just two precautions. 
    
    Firstly, notice that our setting is coherent with the one of \cite[§1.6]{Howard04:heegner-point-kolyvagin-system}, except that our sets of primes $\mathcal{P}'_i$ do not eventually contain \emph{all} primes whose Frobenius is conjugated with $\tau_c$ in $\Gal(K(\T^\ac_\sfrak/p^i)/\Q)$. In \emph{loc.~cit.}, this fact is only used in the proof of \cite[Lemma 1.6.2]{Howard04:heegner-point-kolyvagin-system}. We replace this with Lemma \ref{lem:playing-with-primes}, where we use a finer argument to obtain the same result.

    Secondly, we can safely replace Howard's Kolyvagin system with our modified universal Kolyvagin system $\boldsymbol{\gk}_\sfrak$ in the arguments of \cite[§1.6]{Howard04:heegner-point-kolyvagin-system}. Indeed, one only needs the property that $\loc_\gl(\boldsymbol{\gk}_\sfrak(n\ell)_{i})=0$ if and only if $\loc_\gl(\boldsymbol{\gk}_\sfrak(n)_{i})=0$, which is granted by property (K2) in Definition \ref{dfn:kolyvagin-system}. Here, $n\ell\in\mathcal{N}'$, $\gl=(\ell)$ and $\boldsymbol{\gk}_\sfrak(\cdot)_{i}$ is the reduction of $\boldsymbol{\gk}_\sfrak(\cdot)$ modulo $p^i$. Having this, the proof of \cite[Lemma 1.6.4]{Howard04:heegner-point-kolyvagin-system} goes along word by word.
\end{proof}

Denote by $\bar{\T}^\ac$ the residual $G_\Q$-representation of $\T^\ac$ and, if $M$ is any module with an action of the complex conjugation $\tau_c$, write $M^\pm$ for the submodules of $M$ where $\tau_c$ acts as $\pm 1$, respectively.

\begin{lemma}\label{lem:playing-with-primes}
     Let $c^+\in H^1(K,\bar{\T}^\ac_\sfrak)^+$ and $c^-\in H^1(K,\bar{\T}^\ac_\sfrak)^-$. Then, for every $i\gg 0$ there are infinitely many primes $\ell\in\mathcal{P}_i'$ such that if $c^\pm \ne 0$, then $\loc_\gl(c^{\pm})\ne 0$, where $\gl=(\ell)$ is the prime of $K$ above $\ell$.
\end{lemma}
\begin{proof}
    We assume that both $c^\pm$ are nonzero, the proof of the other cases being analogous. For simplicity, let $T:=\T^\ac_\sfrak$ and $\bar{T}=\bar{\T}^\ac_\sfrak$. Let $F/\Q$ be the extension of \cite[Hypothesis H.2]{Howard04:heegner-point-kolyvagin-system} and let $L_i = K(T/p^i)$, for every $i\in\Z_{>0}$. Since $L_i$ is contained in $F(\mu_{p^{\infty}})$, the restriction
    \begin{equation*}
        H^1(K,\bar{T})\to H^1(L_i,\bar{T})^{\Gal(L_i/K)}\cong\Hom(G_{L_i},\bar{T})^{\Gal(L_i/K)}
    \end{equation*}
    is an injection for any $i\ge1$. Denote by $\psi_i^+$ and $\psi_i^-$ the non-zero homomorphisms of $\Hom(G_{L_i},\bar{T})$ corresponding to $c^+$ and $c^-$. Let $\tilde{L}_i$ be the smallest extension of $L_i$ that is cut out by $\psi_i^+$ and $\psi_i^-$ and is Galois over $\Q$, and let $G_i:=\Gal(\tilde{L}_i/L_i)$.

     Note that $\tau_c$ acts on $G_i$ by conjugation, determining the eigenspaces $G_i^+$ and $G_i^-$ associated with the eigenvalues $1$ and $-1$. Set
    \begin{equation*}
        M_i:=\Hom(G_i,\bar{T})^{\Gal(L_i/K)}
    \end{equation*}
    and call $\bar{\psi}_i^{\pm}$ the elements of $M_i$ determined by $\psi_i^{\pm}$, respectively. There is also an action of $\tau_c$ on $M_i$ and, by hypothesis, $\bar{\psi}_i^{\pm}\in M_i^{\pm}$.

     We now claim that the maps $\bar{\psi}_i^\pm$ are both non-zero on $G_i^+$. This is because each map $\bar{\psi}_i^\pm$ factors through the maximal abelian $p$-primary quotient of $G_i$, which splits as the sum of the two $\pm 1$-eigenspaces for the action of $\tau_c$. 
    Hence, if we suppose that $\bar{\psi}_i^\pm(G_i^+)=0$, then $\bar{\psi}_i^\pm(G_i)=\bar{\psi}_i^\pm(G_i^-)$ is contained in $\bar{T}^\mp$, which is 1-dimensional. Since $\bar{\psi}_i^\pm$ is non-zero and fixed by $\Gal(L_i/K)$, it follows that $\bar{\psi}_i^\pm(G_s)$ spans a non-zero proper $(\calR/\m_\calR)$-submodule of $\bar{T}$ stable under the action of $G_\Q$, which contradicts Assumption \ref{ass:residual-representation-irreducible}. Then, it follows that we can find
    \begin{equation*}
        g\in G_i^+\quad\text{such that}\quad \bar{\psi}_i^\pm(g)\ne 0.
    \end{equation*}
     Notice now that for any $i\ge1$ the (nontrivial) homomorphism $\psi_{i+1}^{\pm}$ coincides with the restriction of $\psi_i^\pm$ to $G_{L_{i+1}}$, so that $\ker \psi_{i+1}^\pm=\ker \psi_i^\pm\cap G_{L_{i+1}}$. Therefore, a careful analysis of the Galois groups involved yields that $L_{i+1}\tilde{L}_i= \tilde{L}_{i+1}\ne L_{i+1}$, hence $G_{i+1}$ is isomorphic to a nonzero subgroup of $G_i$. Since $G_1$ is finite, there is an $i_0$ such that $G_{i+1}\cong G_i$ for all $i\ge i_0$, which yields also that $L_i=L_{i+1}\cap \tilde{L}_i$ and in particular that
    \begin{equation}\label{equ:splitting-Galois-groups-applications}
        \Gal(\tilde{L}_{i+1}/L_i)=\Gal(L_{i+1}/L_i)\times G_i.
    \end{equation}

    Suppose therefore that $i\ge i_0$ and let $\ga\in 1+p^i\Zp$. By Assumption \ref{ass:big-image} and the splitting of \eqref{equ:splitting-Galois-groups-applications}, it follows that there is an element $\gs_\ga\in G_{L_i}$ that fixes $\tilde{L}_i$ and such that the image of $\gs_\ga$ in $\Aut(T_{i+1})$ is the scalar $\ga$. Define $\mathcal{Q}_i(\ga)$ to be the set of primes $\ell$ that are unramified and whose Frobenius $\Fr_\ell$ is conjugated with $\tau_c g \gs_\ga$ in $\Gal(\tilde{L}_{i+1}/\Q)$. Notice that, in particular, $\Fr_\ell|_{\tilde{L}_i}$ is conjugated with $\tau_c g$ and $\Fr_\ell|_{L_{i+1}}$ is conjugated with $\tau_c \sigma_{\alpha}$.

     By Chebotarev's density theorem, the set $\mathcal{Q}_i(\ga)$ is infinite. Our aim is to find a suitable $\ga$ such that $\mathcal{Q}_i(\ga)\subseteq\mathcal{P}_i'$ and such that the primes of $\mathcal{Q}_i(\ga)$ satisfy the claim of the lemma.
    
   Choose a prime $\tilde{\gl}_i$ of $\tilde{L}_i$ above $\ell\in \mathcal{Q}_i(\ga)$ and set $\gl_i=\tilde{\gl}_i\cap L_i$ and $\gl=\gl_i\cap K$. Notice that $\gl$ splits completely in $L_i$, because $\Fr_{\lambda}$ acts trivially on $T_i$. Moreover $\Fr_{\tilde{\gl}_i/\gl_i}=(\Fr_{\tilde{\gl}_i/\ell})^2=\tau_c \, g\,\tau_c\, g=g^2\in G_i$, using the fact that $g\in G_i^+$, and therefore
    \begin{equation*}
        \bar{\psi}^{\pm}(\Fr_{\tilde{\gl}_i/\gl_i})
        =\bar{\psi}^{\pm}(g^2)=2\bar{\psi}^{\pm}(g)\ne 0,
    \end{equation*}
    since $p\ne 2$. This shows that the restriction of $\bar{\psi}^{\pm}$ to $\Gal\big((\tilde{L}_i)_{\tilde{\gl}_i}/(L_i)_{\gl_i}\big)$ is non-zero and hence, since $(L_i)_{\gl_i}=K_\gl$, we obtain that $\loc_\gl c^{\pm}$ is non-zero.

   Eventually, one can argue as in the proof of Lemma \ref{lem:finer-choice-of-primes} to find a suitable $\ga\in 1+p^i\Z_p$ such that $\mathcal{Q}_i(\ga)\subseteq \mathcal{P}_i'$, concluding the proof.
\end{proof}




\printbibliography

\end{document}